%% file: Mmt-stability-AMS.tex
\documentclass[11pt]{amsart}
\usepackage[left=1in,top=1in,right=1in,bottom=1in]{geometry}
\usepackage{amsmath, amsthm, amssymb, enumitem, bm}
\usepackage[mathscr]{eucal}
\usepackage[square, numbers]{natbib}

\usepackage{macro}
\usepackage{url}
\usepackage{braket}
\usepackage{dsfont}
\usepackage[textwidth=1in, textsize=small]{todonotes}
\usepackage{easylist}

\newcommand{\const}{\mathscr{C}}
\newcommand{\cnst}{\mathscr{B}}

\newcommand{\pwr}{\theta}
\newcommand{\pow}{\varsigma}
\newcommand{\lst}{\eta}
\newcommand{\hft}{\varsigma}
\newcommand{\jdiff}{\Xi}
\newcommand{\inpo}{{x_0}}

\newtheorem{theorem}{Theorem}[section]
\newtheorem{lemma}[theorem]{Lemma}
\newtheorem{assumption}[theorem]{Assumption}

\newtheorem{definition}[theorem]{Definition}
\newtheorem{proposition}[theorem]{Proposition}
\newtheorem{remark}[theorem]{Remark} 

\newtheorem{corollary}[theorem]{Corollary}
\newtheorem{example}[theorem]{Example}

\title[Moment stability of stochastic processes]{Moment stability of stochastic processes with applications to control systems}

\author{Arnab Ganguly}
\address{Department of Mathematics, Louisiana State University, USA.}
\email{aganguly@lsu.edu}
\thanks{Research of A. Ganguly is supported in part by NSF DMS - 1855788 and Louisiana Board of Regents through the Board of Regents Support Fund (contract number: LEQSF(2016-19)-RD-A-04)}

\author{Debasish Chatterjee }
\address{Indian Institute of Technology Bombay, Systems \& Control Engineering, India.}
\email{ dchatter@iitb.ac.in}



\begin{document}

\begin{abstract}
We establish new conditions for obtaining uniform bounds on the moments of discrete-time stochastic processes. Our results require a weak negative drift criterion along with a state-dependent restriction on the sizes of the one-step jumps of the processes. The state-dependent feature of the results make them suitable for a large class of multiplicative-noise processes. Under the additional assumption of Markovian property, new result on ergodicity has also been proved.  There are several applications to iterative systems, control systems, and other dynamical systems with state-dependent multiplicative noise, and we include  illustrative examples to demonstrate  applicability of our results.
\end{abstract}
\date{\today}

\maketitle

\input{intro}

\input{results}

\input{applications}

\bigskip
\bigskip

\bibliographystyle{plainnat}
\bibliography{mmt-ref}

\end{document}

%% file: intro.tex
\section{Introduction}
The paper studies stability properties of a general class of discrete-time stochastic systems. Assessment of stability of dynamical systems is an important research area which has been studied extensively over the years. For example, in  control theory a primary objective is to design suitable control policies which will ensure appropriate stability properties (e.g. bounded variance) of the underlying controlled system.  There are various notions of stability of a system. In mathematics,  stability often refers to  equilibrium stability, which,  for deterministic dynamical systems, is mainly concerned with qualitative behaviors of the trajectories of the system that start near the equilibrium point. For the stochastic counterpart, in Markovian setting it usually involves study of  existence of  invariant distributions and associated convergence and ergodic properties.  A comprehensive source of results  on different  ergodicty properties for discrete-time Markov chains using Foster-Lyapunov functions is \cite{ref:MeyTwe-09} (also see the references therein). Several extensions of such results have since then been explored quite extensively in the literature (for example, see \cite{MT93-1, MT93-2}).   Another important book in this area is \cite{HL03}, which uses expected occupation measures of the chain for identifying conditions for stability.

The primary objective of the paper is to study moment stability, which concerns itself with uniform bounds on moments of a general stochastic process $X_n$ or, more generally, on expectations of the form  $\EE (V(X_n))$ for a given function $V$. This is a bit different from the usual notions of stability in  the Markovian setting as mentioned in the previous paragraph, but they are not unrelated. Indeed, if the process $\{X_n\}$ has certain underlying Lyapunov structure a strong form of Markovian stability holds which in particular implies moment stability.
% is Markov moment stability is a consequence of a result based on Foster-Lyapunov criterion establishing  Harris ergodicity of $\{X_n\}$. \cite{ref:MeyTwe-09}.
%  At this point it is informative to compare our results with typical Foster-Lyapunov techniques as popularized by \cite{ref:MeyTwe-09} to point out the difference in scope of their applications. 
The  result, which is based on Foster-Lyapunov criterion, can be described as follows. Given a Markov chain $\{X_n\}_{n\in\Nz}$ taking values in a Polish space $\SC{S}$ with a transition probability kernel $\tk$,  suppose there exists a non-negative measurable function $u:\SC{S}\rt [0,\infty)$, called a Foster-Lyapunov function, such that the process $\{u(X_n)\}_{n\in\Nz}$ possesses has the following negative drift condition:
for some constant $b\geq 0, \ \theta>0$, a set $A \subset \SC{S}$, and a function $V: \SC{S} \rt [0,\infty)$ 
\begin{align}
	\label{FL-drift}
\EE\lf[u(X_{n+1}) - u(X_n) | X_n=x \ri] \equiv	\int_{\SC{S}} \tk(x,dy)u(y) -u(x) \leq -\theta V(x) + b\indic{\{x\in A\}}.
\end{align} 
If the set $A$ is petite, (which, roughly speaking, are the sets that have the property that any set $B$ is `equally accessible' from any point inside the petite set - for definition and more details, see \cite{MT92, ref:MeyTwe-09}),  the process $\{X_n\}$ has a unique invariant distribution $\pi$ and also $\pi(V) = \int_{\SC{S}} \pi(dx) V(x) < \infty$. Moreover, under aperiodicity, it can be concluded that  the chain is Harris ergodic, that is,
$$\|\tk^n(x,\cdot) - \pi\|_V \rt 0, \quad \mbox{ as } n \rt \infty,$$
where $\|\cdot\|_V$ is the $V$-norm (see, the definition at the end of introduction) \cite[Chapter 14]{ref:MeyTwe-09}.  In particular, one has $\EE[V(X_n)] \rt \pi(V)$ as $n\rt \infty$ (which of course implies boundedness of $\EE[V(X_n)]$). Thus for a Markov process $\{X_n\}$, one way to get a uniform bound on $V(X_n)$ is to find a  Foster-Lyapunov function $u$ such that \eqref{FL-drift} holds.

The objective of the first part of the paper is to explore scenarios where a strong negative drift condition like \eqref{FL-drift} does not hold or at least such a Lypaunov function is not easy to find for a specific $V$. We do note that the required conditions in our results are formulated in terms of the target function $V$ itself. One pleasing aspect of this feature is that search for a suitable Lyapunov function $u$ is not required for applying these results.

Our main result, Theorem \ref{t:maintheorem}, deals with the general regime where the state process \(\{X_n\}\) is a general stochastic process and not necessarily Markovian. While past study on stability mostly concerns homogeneous Markov processes, the literature  in the case of more general processes including non-homogeneous Markov processes and processes with long-range dependence  is rather limited. The starting point in Theorem  \ref{t:maintheorem} is a weaker negative drift like condition:
\begin{align}
\EE\lf(V(X_{n+1}) - V(X_n) | \SC{F}_n\ri) \leq - A, \quad X_n \notin \SC{D}
\label{th2:negdrft}
\end{align}
which, if $X_n$ is a homogeneous Markov chain, is of course equivalent to $\tk V(x) - V(x) \leq -A$ for $x$ outside $\SC{D}$. %But the similarity with the negative drift like condition in Foster Lypaunov techniques of \cite{ref:MeyTwe-09} ends here. 
As could be seen by comparing \eqref{th2:negdrft} with \eqref{FL-drift} that even in the Markovian setting, the results of \cite[Chapter 14]{ref:MeyTwe-09} will not imply $\sup_n\EE(V(X_n)) < \infty.$ In fact, condition \eqref{th2:negdrft} is not enough to guarantee such an assertion even in a deterministic setting. For example, consider the sequence $\{x_n\}$ on $\N$ defined by
\begin{align*}
x_{n+1} = \begin{cases} 
x_n-1, &\quad \text{ if } x_n >1\\
n+1, &\quad \text{ if } x_n = 1.
\end{cases}
\end{align*}
Clearly, $\sup_{n\geq 1} x_n = \infty,$ even though the negative drift condition is satisfied for $\SC{D}=\{1\}$.   But we showed in Theorem \ref{t:maintheorem} that under a state-dependent restriction on the conditional moments of $V(X_{n+1})$ given $\sigalg_n$ (see Assumption \ref{cond_stability2} for details),  the desired uniform moment bound can be achieved. Note that the above sequence $\{x_{n+1}\}$ fails \ref{t:maintheorem:bnds} of Assumption \ref{cond_stability2} but satisfies the other two conditions.

In the (homogeneous) Markovian framework, Theorem  \ref{t:maintheorem} leads to a new result (c.f. Theorem \ref{MP-stab}) on Harris ergodicity of Markov chains which will be useful in occasions when Foster-Lyapunov drift criterion in form of \eqref{FL-drift} does not hold.  Importantly, Theorem \ref{MP-stab} does not require $\SC{D}$ to be petite or prior checking of aperiodicity of the chain.

%
%It should be noted that in neither of the above two results the set $K$ is required to be petite which is a typical assumption in the typical Foster Lyapunov type results like \eqref{FL-drift}. If the Markov chain is $\phi$-irreducible and weak-Feller, then all compact sets are petite, but in other cases checking petiteness can sometimes be  a bit of extra work. But of course, conditions like \eqref{FL-drift} intend to give more information (e.g, positive Harris recurrence, ergodicity) than uniform moment bounds. However, we do note in Remark \ref{rem-stat} that a uniform moment-like bound as obtained in Theorems \ref{first_th} and \ref{t:maintheorem} does imply existence of stationary distribution if we impose additional assumptions like  weak-Feller property. Furthermore, the stationary distribution will be unique under the assumption of suitable notions of irreducibility.

Theorem \ref{t:maintheorem} is partly influenced by a result of Pemantle and Rosenthal \cite{ref:PemRos-99} which established a uniform bound on $\EE(V^r(X_n))$ under \eqref{th2:negdrft} and the additional assumption of  a {\em constant bound} on conditional $p$-th moment of one-step jumps of the process given $\sigalg_n$, that is, $\EE\lf[|V(X_{n+1}) -V(X_n)|^p|\sigalg_n\ri]$. However, for a large class of stochastic systems  the latter requirement of a  uniform bound on conditional moments of jump sizes cannot be fulfilled. In particular, our work is motivated by some problems on stability about a class of stochastic  systems with multiplicative noise where such conditions on one-step jumps are almost always state-dependent and can never be bounded by a constant.
  Our work  generalized the result of \cite{ref:PemRos-99} in two important directions - it uses a different ``metric" to control the one step jumps and it allows such jumps to be bounded by a suitable {\em state dependent} function. 
    Specifically, instead of $\EE\lf[|V(X_{n+1}) -V(X_n)|^p|\sigalg_n\ri]$ we control the {\em centered} conditional $p$-th moment of $V(X_{n+1})$, that is, $\EE\lf[\lf|V(X_{n+1}) -\EE(V(X_{n+1})|\sigalg_n\ri|^p\Big|\sigalg_n\ri]$, in a state-dependent way. The latter quantity can be viewed as a distance between the actual position at time $n+1$, $V(X_{n+1})$, and the expected position at that time given the past information, $\EE(V(X_{n+1})|\sigalg_n)$, while \cite{ref:PemRos-99}  uses the distance between actual positions at times $n+1$  and  $n$. These extensions require a different approach involving different auxiliary estimates. The advantages of this new `jump metric' and the state dependency feature have been discussed in detail after the the proof of Theorem \ref{t:maintheorem}. Together, they significantly increase applicability of our result to a large class of stochastic systems.
    
     This is demonstrated in Section \ref{sec-app}, where  a broad class of systems with multiplicative noise is studied and new stability results (see Proposition \ref{prop:switch} and Corollary \ref{cor-erg-mark}) are obtained. This, in particular, includes stochastic switching systems and Markov processes of the form $X_{n+1} = H(X_n) + G(X_n)\xi_{n+1}.$ The last part of this section is devoted to the important problem of stabilization of stochastic linear systems with bounded control inputs. The problem of interest here is to find conditions which guarantee $L^2$-boundedness of a stochastic linear system of the form $X_{n+1} = AX_n+Bu_n+\xi_{n+1}$ with bounded control input. This has been studied in a previous work of the second author (see \cite{ref:RamChaMilHokLyg-10} and references therein for more background on the problem), and it has been shown that when $(A,B)$ is stabilizable, there exists a $k$-history dependent control policy which assures bounded variance of such system provided the norm of the control is sufficiently large. This upper bound on the norm of the control is an artificial obstacle on its design, and it has been conjectured in \cite{ref:RamChaMilHokLyg-10} that it is not required although a proof couldn't be provided. Here we show that this conjecture is indeed true (c.f. Proposition \ref{prop:control}), and the artificial restriction on the control norm can be lifted largely owing to the new ``metric" in Theorem \ref{t:maintheorem}. In fact, as Proposition \ref{prop:switch} and Corollary \ref{cor-erg-mark} indicate this stabilization result can be easily extended to cover more general classes of stochastic control systems including the ones with multiplicative noise.
  %can be expected, such a generalization is nontrivial and requires construction of certain auxiliary  processes and  careful estimation of their moments. 

The  article is organized as follows. The mathematical framework and the main results are described in Section \ref{sec-main2}.  Section \ref{sec-app} discusses potential applications of our results for a large class of stochastic systems including switching systems, multiplicative Markov models,  which are especially relevant to control theory. 

\np
{\em Notation and terminology:} For a probability kernel $P$ on $\SC{S}\times \SC{S}$, and a function $f: \SC{S} \rt [0,\infty)$, the function $Pf: \SC{S} \rt [0,\infty)$ will be defined by $Pf(x) = \int_{\SC{S}} f(y)P(x,dy)$. In similar spirit, for a  measure $\mu$ on $\SC{S}$, $\mu(f)$ will be defined by $\mu(f) = \int_{\SC{S}}f(x)\mu(dx).$ For a signed measure, $\mu$, on $\SC{S}$. the corresponding  total variation measure is denoted by  $|\mu| = \mu^++\mu^-$, where $\mu=\mu^+-\mu^-$ as per  the Jordan decomposition. If $\mu  =\nu_1 - \nu_2$, where $\nu_1$ and $\nu_2$ are probability measures, the total variation distance $\|\nu_1 - \nu_2\|_{TV}$ is given by
$$\|\nu_1 - \nu_2\|_{TV} = |\mu|(\SC{S}) = 2\sup_{A \in \SC{B}(\SC{S})}|\nu_1(A) - \nu_2(A)|.$$
More generally, if $g: \SC{S} \rt [0,\infty)$ is a measurable function, the {\em $g$-norm} of $\mu = \nu_1-\nu_2$ is defined by $\|\mu\|_g = \sup\{|\mu(f)|: f \text{ measurable and }0\leq f\leq g\}$ 

%For a function $h:\SC{S} \rt \R$, $ \lim_{x\rt \infty} h(x)=\infty$ is interpreted in the following natural sense: for every $\e>0$, there exists a bounded set $C\subset \SC{S}$ such that $ \sup_{x\notin C}|h(x)-L|\leq \e.$

  Throughout, we will work on an abstract probability space $(\Omega, \SC{F}, \PP)$.  $\EE$ will denote the expectation operator under $\PP$. In context of the process $\{X_n\}$, $\EE_x$ will denote the conditional expectation given $X_0=x$.

%% file: results.tex
	\section{Mathematical framework and main results}\label{sec-main2}

The  section presents two main results, Theorem \ref{t:maintheorem} on uniform bounds on functions of a general stochastic process and Theorem \ref{MP-stab} on ergodicity properties in the homogeneous Markovian setting.  The mathematical framework pertains to a stochastic process  $\{X_n\}$ taking values in a topological space \(\SC{S}\) and involves negative drift conditions outside a set $\SC{D}$, together with a state-dependent control on the size of one-step jumps of $\{X_n\}$.
%In particular, the jumps outside \(K\) are  allowed to grow in a suitable way with the distance from \(K\). %In this sense the drift condition in Theorem \ref{t:maintheorem} is stronger than the one in \cite[Theorem 2.1]{ref:PemRos-99}, but the jump condition is weaker.

\subsection{Uniform bounds for moments of stochastic processes} \label{sec:unifbd}
	\begin{assumption}\label{cond_stability2}
		There exist measurable functions $V:\SC{S}\rt [0,\infty), \ \jumpbnd:\SC{S}\rt [0,\infty)$, and a set $\SC{D}\subset \SC{S} $ such that
		\begin{enumerate}[label={\rm (\ref{cond_stability2}-\alph*)}, widest=b, leftmargin=*, align=left, nolistsep]
		\item \label{t:maintheorem:driftbnd} for all $n\in N$, 
		$$\EE_\inpo[V(X_{n+1}) - V(X_n)\mid \sigalg_n] \lq -A \mbox{\ \ on\ \ } \{X_n \notin \SC{D}\};$$
		\item \label{t:maintheorem:jumpbnd} for all $n\in \N$ and some $p > 2$, $\jdiff_n$,   the centered conditional $p$-th moment of $V(X_{n+1})$ given $\sigalg_n$, satisfy
		$$\jdiff_n \doteq   \EE_\inpo\Big[|V(X_{n+1}) - \EE(V(X_{n+1})|\sigalg_n)|^p \Big| \sigalg_n\Big] \lq \jumpbnd(X_n),$$
		where $\jumpbnd(x) \leq \const_{\jumpbnd}(1+V^s(x))$ for some $0\leq s<p/2-1$ and some constant $\const_{\jumpbnd}>0$.
	%	\item \label{cond:growth} for some constant $C_r\geq 0$
	%	$$r(x)\lq C_r(1+V(x)), \quad \mbox{ for all}\ \ x\in \SC{S}.$$
		\item \label{t:maintheorem:bnds} $\sup_{x\in  \SC{D}}V(x)<\infty,$ and for some constant $\bar\cnst_0(\inpo)$,  
		$$\EE_\inpo\lf[\lf(\EE[V(X_{n+1})|\sigalg_n]\ri)^p\indic{\{X_n \in  \SC{D}\}} \ri] < \bar\cnst_0(\inpo).$$
%		\item $\dst M\equiv \sup_{x\in K}\EE_x\lf(V(X_1)\indic{\{X_1\notin K\}}\ri) <\infty.$
%		\item $\dst\bar{\L}\equiv \sum_{n\geq 0}n\l_n <\infty.$
		\end{enumerate}
		\end{assumption}

\begin{theorem}
		\label{t:maintheorem}
		Suppose that Assumption \ref{cond_stability2} holds for the process $\{X_n\}$ with $X_0=\inpo$.
%		 and suppose $\jumpfun$ in the assumption is such that $\jumpfun(x) \leq \const_{\jumpfun}(1+V^s(x))$ for some $s<p/2-1$ and some constant $\const_{\jumpfun}>0$.
           Then  
		    $$ \cnst_r(\inpo) \doteq \sup_{n\in\N} \EE_\inpo\bigl[V(X_n)^r\bigr] < \infty,$$
		    for any $0\leq r<  \pow(s,p),$
		    where 
		    $$\pow(s,p) = \begin{cases} p\lf(1 - \f{s}{p-2}\ri)-1,& \quad \text{ for } s \in [0, (p-2)^2/2p) \cup [1-2/p,\ p/2-1), \text{ when } \ 												2<p<4; \\
		    & \quad \text{ for all } s \in [0,p/2-1), \text{ when }  \ p \geq 4;\\
		    p-2, &\quad   \text{ for }\ (p-2)^2/2p \leq s <  1-2/p, \text{ when } \ 2<p<4. \\
	%	    p\lf(1 - \f{s}{p-2}\ri)-1, & \quad \text{ for } \ p \geq 4.
		    \end{cases}
		     $$
		    
%	In particular, the result holds if $\varphi$ is such that $\varphi(x)/V^r(x) \rt 0$ as $x \rt \infty.$	
\end{theorem}	

\begin{remark}
{\rm
\hs{.2cm} 
\begin{itemize}
\item  The proof is a combination of Proposition \ref{t:maintheorem2} and Proposition \ref{t:maintheorem3}.  Proposition \ref{t:maintheorem2} first establishes a weaker version of the above assertion by showing that $\sup_{n\in\N} \EE_\inpo\bigl[V(X_n)^r\bigr]  < \infty$, for all $r<p/2-1$. However, extension of the result from there to all $r <\pow(s,p)$ (notice that $\pow(s,p) \geq p/2-1$) requires a substantial amount of extra work and is achieved through Proposition \ref{t:maintheorem3}.

\item Note that \ref{t:maintheorem:bnds} is implied by the simpler condition: $\EE_\inpo[V(X_{n+1})|\sigalg_n] \leq \bar\cnst_0$ on $\{X_n \in \SC{D}\}$ for some constant $\bar\cnst_0$.
\end{itemize}
}
\end{remark}

\begin{proof}[Proof of Theorem \ref{t:maintheorem}] 
From Proposition \ref{t:maintheorem2} and the growth assumption on $\jumpbnd$, it follows that for any $1\leq \theta < (p-2)/2s$, $\sup_n \|\jdiff_n\|_\theta \leq \sup_n\lf(\EE_\inpo(\jumpbnd^\theta(X_n))\ri)^{1/\pwr} < \infty$,  where $\|\cdot\|_\theta$ is the $\SC{L}^\theta(\Omega, \PP)$-norm (c.f. Proposition \ref{t:maintheorem3}).    The result now follows from Proposition \ref{t:maintheorem3} by letting $\theta \rt (p-2)/2s-.$ If $s=0$, that is, $\jdiff_n \leq \const$, for some constant $\const$ a.s., we take $\theta=\infty$ in Proposition \ref{t:maintheorem3}.
%
%We first consider the case when $p\geq 4$. Fix an $r < \pow(s,p) =p\big(1-s/(p-2)\big)-1.$ Define $\theta$ by
%$r < p\lf(1-\f{1}{2\theta}\ri)-1$, and therefore, for any 1\leq $\theta< \f{p-2}{2s}$, or equivalently, $s\theta < p/2-1.$ Consequently,
%by Proposition \ref{t:maintheorem2}, 
%$$\|\jdiff_n\|_\pwr \leq \lf( \EE\lf[\jumpfun^{\theta}(X_n)\ri]\ri)^{1/\theta} \leq \const_\jumpfun \lf( \EE\lf[1+V^{s}(X_n)\ri]^\theta\ri)^{1/\theta} < \infty.$$
%Moreover, notice that $\f{p}{2\theta} < p-r-1.$ Thus if $\f{p}{2\theta} \geq 1$, then $p-r-1>1$. If $\f{p}{2\theta} < 1$, then 
%$$\f{p}{2\theta^*} = \f{p}{2} - \f{p}{2\theta} > 2-1=1.$$
%
%We now consider the case of $2<p<4$. In this range,
%$$\f{(p-2)^2}{2p} < 1-\f{2}{p}.$$
%Suppose $(p-2)^2/2p \leq s \leq  1-2/p$.
%\begin{align*}
%s \leq 1-2/p \quad \Leftrightarrow \quad p-2 \leq p(1-s/(p-2)) -1
%\end{align*}
%Thus if $p(1-s/(p-2)) -r \geq p-r-1 >1.$ Also note that 
%$$ (p-2)^2/2p \leq s\quad \Leftrightarrow \quad \f{p}{2}(1-2s/(p-2))\leq 1.$$
%Also, this shows that for $ s \in [0, (p-2)^2/2p)$, $\f{p}{2}(1-2s/(p-2))> 1$. Next notice that for $2<p\leq 4$
%
\end{proof}

At this stage it is instructive to compare Theorem \ref{t:maintheorem} with \cite[Theorem 1]{ref:PemRos-99} and precisely  note some of  the improvements the former offer. The first significant extension is that Theorem \ref{t:maintheorem} allows the jump sizes in \ref{t:maintheorem:jumpbnd} to be state dependent
whereas,  \cite{ref:PemRos-99} requires
\begin{align}
\label{eq:pero}
\EE_\inpo\lf[|V(X_{n+1}) - V(X_n)|^p | \SC{F}_n\ri] \leq B, \tag{\dag}
\end{align}
for some constant $B>0$. The resulting benefits are obvious as it allows the result in particular to be applicable to large class of multiplicative systems of the form
$$X_{n+1} = H(X_n) + G(X_n)\xi_{n+1},$$
which \cite[Theorem 1]{ref:PemRos-99} will not cover.
The second notable distinction is in the `metric' used in \ref{t:maintheorem:jumpbnd} in controlling jump sizes : while \cite{ref:PemRos-99} involves $\EE\lf[|V(X_{n+1}) - V(X_n)|^p | \SC{F}_n\ri]$, our result only requires controlling the centered conditional $p$-th moments of $V(X_{n+1})$ given $\sigalg_n$, namely, $\EE_x\lf[\big|V(X_{n+1}) - \EE[V(X_{n+1})|\sigalg_n]\big|^p\Big|\sigalg_n\ri]$.  Of course, the latter leads to weaker hypothesis as
$$\EE_\inpo\lf[\big|V(X_{n+1}) - \EE[V(X_{n+1})|\sigalg_n]\big|^p\Big|\sigalg_n\ri] \leq 2^{p}\EE_\inpo\lf[|V(X_{n+1}) - V(X_n)|^p | \SC{F}_n\ri]. $$
It is important to emphasize the advantages of the weaker hypothesis as   the condition in \eqref{eq:pero}
precludes it from being applicable even to some additive models. To illustrate this with a simple example, consider a $[0,\infty)$-valued process $\{X_n\}$ given by
$$X_{n+1} = X_n/2 + \xi_{n+1}, \quad X_0 \geq 0,$$
where $\xi_n$ are  $[0,\infty)$-valued random variables with $\mu_p =\sup_n\EE(\xi_n^p) < \infty$ for $p>2$.  Since 
$X_{n+1} - X_n = -X_n/2+\xi_{n+1}, $ clearly the negative drift condition (c.f \ref{t:maintheorem:driftbnd}) holds with $V(x) = |x|$. 
but for the jump sizes  we can only have
\begin{align*}
\EE_\inpo\lf[|X_{n+1} - X_n|^p | \sigalg_n\ri] = O(X_n^p).
\end{align*}
This means that \cite[Theorem 1]{ref:PemRos-99} cannot  be used  to get $\sup_n \EE_x(X_n) < \infty$ for this  simple additive system -  a fact which easily follows from an elementary iteration argument (note, $\EE_x(X_n) \stackrel{n\rt\infty}\rt 2\mu_1$).  On the other hand, our theorem clearly covers such cases as  
\begin{align*}
\EE_\inpo\lf[|X_{n+1} - E\lf(X_{n+1}  | \sigalg_n\ri)|^p\Big|\sigalg_n\ri] \leq \bar \mu_p, \quad \bar \mu_p = \sup_n\EE|\xi_n - \EE(\xi_n)|^p.
\end{align*}
%The benefit of our `jump metric' in covering a broader class of stochastic systems is particularly prominent as it can be seen that the state-dependency feature of our theorem  was not needed to get a moment bound for this simple additive system.
It should actually be noted that had Theorem \ref{t:maintheorem} simply controlled the jump sizes by imposing the more restrictive condition,
 $\EE\lf[|X_{n+1} - X_n|^p | \sigalg_n\ri] \leq \jumpbnd(X_n)$, the state-dependency feature was not enough to salvage the moment bound of the above additive system (because of the requirement $\jumpbnd(x) = O(V^s(x))$ for $s<p/2-1$). It is interesting to note that  the  results of \cite{ref:MeyTwe-09} based on Foster-Lyapunov drift conditions also cannot directly be used in this simple example, as  $\{X_n\}$ is not necessarily Markov (since the $\xi_n$ are not assumed to be i.i.d). To summarize, the weaker jump metric coupled with state dependency feature makes Theorem \ref{t:maintheorem} a rather powerful tool in understanding  stability for a broad class of stochastic systems. Some important results in this direction for switching systems have been discussed in the application section.
% 
%Theorem \ref{t:maintheorem} explores the case when $\SC{L}^1(\Omega, \PP)$-norm of 
% $\jdiff_n \equiv   \EE\lf[|V(X_{n+1}) - \EE(V(X_{n+1})|\sigalg_n)|^p | \sigalg_n\ri]$
% is bounded by that of $\jumpbnd(X_n)$.
%
%If $\jumpbnd(x) \equiv C$, a constant, then Theorem \ref{t:maintheorem} ensures uniform bound for $\EE_x\bigl[V(X_n)^r\bigr]$ for $r<p/2-1$. It is natural to ask if in such special cases we can optimize this result to get the desired bounds for larger $r$. But indeed as in \cite{ref:PemRos-99} which under \eqref{eq:pero} achieved uniform bounds  for $r<p-1$. The following theorem addresses this question in more generality.

The following lemma will be used in various necessary estimates.

		\begin{lemma}
		\label{l:aux1}
			Let \({M_n}\) be a martingale  relative to the filtration \(\{\sigalg_n\}\), 
			\begin{align}\label{jumpcondM}
			\g_n \stackrel{def}=	\EE\bigl[\abs{M_{n+1} - M_n}^p \,\big|\, \sigalg_n\bigr], \quad n\geq 0
			\end{align}
			 $\Theta$ a  non-negative random variable, and $b>0$ a constant. 
			Then for some constants $\const_0$ and $\const_{00}$
			\begin{enumerate}[label={\rm (\alph*)}, widest=b, leftmargin=*, align=left, nolistsep]
			\item $\dst \EE\lf[|M_n-M_k|^p|\SC{F}_k\ri]  \leq \const_0(n-k)^{\f{p}{2}-1} \sum_{m=k}^{n-1} \EE[\g_m|\SC{F}_k].$
			\item $\dst \EE\lf[(|M_n-M_k|+\Theta)^r \indic{\{|M_n-M_k|+\Theta) >b \}}|\SC{F}_k\ri]  \lq \const_{00}\lf((n-k)^{\f{p}{2}-1} \sum_{m=k}^{n-1} \EE\lf[\g_m|\sigalg_k\ri] + \EE\lf[|\Theta|^p|\sigalg_k\ri]\ri)b^{r-p}.$

			\end{enumerate}
			\end{lemma}

		\begin{proof}
			Note that by Burkholder's inequality (e.g., see \cite{Pr05}), there exists \(c_p > 0\) such that
			\[
				\EE\bigl[\abs{M_n - M_k}^p\,\big|\, \sigalg_k\bigr] \lq c_p \EE\le[\biggl(\sum_{m = k}^{n-1} \bigl[ \abs{M_{m+1} - M_{m}}^2 \bigr]\biggr)^{p/2}\,\bigg|\,\sigalg_k\ri].
			\]
			Now by H\"older's inequality  and by \eqref{jumpcondM}
			%and conditional Jensen's inequality (in view of convexity of the function \(z \mapsto z^{p/2}\)) on the right-hand side gives
			\begin{align}
			\non
				\EE\bigl[\abs{M_n - M_k}^p\,\big|\, \sigalg_k\bigr]	&  \lq c_p(n-k)^{\f{p}{2}-1} \sum_{m=k}^{n-1} \EE\lf[\abs{M_{m+1} - M_{m}}^p|\SC{F}_k\ri]\\
				\label{burk-ineq}
				& \leq  c_p(n-k)^{\f{p}{2}-1} \sum_{m=k}^{n-1} \EE\lf[\g_m|\SC{F}_k\ri].
			\end{align}

%			
%			Furthermore, by \eqref{jumpcondM}, $\EE[\abs{M_1 - M_0}^p| \SC{F}_0 ] \lq \g(Z_0),$ and also $\|M_0\| \leq B$. Thus, writing
%			$M_n = (M_n-M_1)+(M_1-M_0) +M_0,$ and 
%			 using the inequality $(a+b+c)^p \leq 3^{p-1}(a^p+b^p+c^p)$, we get
%			\begin{align*}
%			\EE[\abs{M_n}^p|\SC{F}_0] & \lq 3^{p-1}\lf( \|B\|^p+\g(Z_0)+c_pn^{\f{p}{2}-1} \sum_{m=1}^{n-1} \EE\lf[\g(Z_m)|\SC{F}_0\ri].\ri).			\end{align*}
	Now observe that for a random variable $Y_n$, by H\"older's inequality and Markov's inequality: $\PP(|Y_n| >b | \SC{F}_k) \leq \EE\lf[|Y_n|^p |\SC{F}_k\ri]/b^p$, we have for $r<p$	and $n\geq k$	
			\begin{align*}
			\EE\lf[|Y_n|^r \indic{\{|Y_n| >b \}}|\SC{F}_k\ri] \leq \ \EE\lf[|Y_n|^p |\SC{F}_k\ri]^{r/p} \PP(|Y_n| >b | \SC{F}_k)^{(p-r)/p}
			         \leq \ \EE\lf[|Y_n|^p |\SC{F}_k\ri]/b^{p-r}.
		 \end{align*}
		Taking $Y_n = |M_n - M_k| + \Theta$ we have
		\begin{align*}
		\EE\lf[|M_n - M_k| + \Theta|^r \indic{\{||M_n - M_k| + \Theta_k| >b \}}|\sigalg_k\ri] \leq \
		     2^{p-1} \lf(\EE\lf[|M_n - M_k|^p|\sigalg_k\ri] + \EE\lf[\Theta^p|\sigalg_k\ri]\ri)/b^{p-r},
		\end{align*}
		and part (b) follows from \eqref{burk-ineq}.
		\end{proof}
	
	We now prove the two propositions which form the backbone of our main result, Theorem \ref{t:maintheorem}.

	\begin{proposition}
	\label{t:maintheorem2}
	Suppose that  Assumption \ref{cond_stability2} holds. Then  for any $0\leq r<  p/2-1,$
	$$ \cnst_r(\inpo) \doteq \sup_{n\in\N} \EE_\inpo\bigl[V(X_n)^r\bigr] < \infty,$$
	
\end{proposition}			
		
	\begin{proof}[Proof of Proposition \ref{t:maintheorem2}] 
%	We first prove that for any $0<r<p/2-1$,
%	$\dst \cnst_r(x) \doteq \sup_{n\in\N} \EE_x\bigl[V(X_n)^r\bigr] < \infty$.

 Fix an $r \in (s,p/2-1)$. Observe that it is enough to prove the result for such an $r$.
Writing $\jumpbnd(x) = \jumpbnd(x) \indic{\{|V(x)| \leq M\}} +(\jumpbnd(x)/V^r(x)) V^r(x) \indic{\{|V(x)| > M\}}$, we can say, because of the growth assumption on $\jumpbnd$ (c.f \ref{t:maintheorem:jumpbnd}), that for every $\vep>0$, there exists a constant $\const_1(\vep)$ such that 
 $\jumpbnd(x) \leq  \const_1(\vep)+\vep V^r(x).$
 
 The  constants appearing in various estimates below will be denoted by $\const_i$'s. They will not depend on $n$ but may depend on the parameters of the system and the initial position $\inpo$.

	Define $\mart_0=0$ and
	\begin{align*}
	\mart_n = \sum_{j=0}^{n-1} V(X_{j+1}) - \EE_\inpo[V(X_{j+1})|\sigalg_j], \quad n \geq 1.
 	\end{align*}
	Then $\mart_n$ is a martingale. Fix \(N\in \N\), and define the last  time $\{X_k\}$ is in $\SC{D}$: 
			\[
				\lst \Let \max\{k\lq N\mid X_k \in \SC{D}\}.
			\]
	Notice that $\{\lst = k\} = \{X_k \in \SC{D}\} \cap \cap_{j>k}^N\{X_j \notin \SC{D}\}$.
%	 where for each $k\geq 0$, the stopping time $\tau^{(k)}$ is defined as
%	\begin{align*}
%	\tau^{(k)} = \inf\{j \geq 1: X_{k+j} \in K\}.
%	\end{align*}
	On $\{\lst = k\}$, for $ k < n \leq N$
	\begin{align} \non
	\mart_n - \mart_k = & V(X_n) - V(X_k) - \sum_{j=k}^{n-1} \lf( \EE_\inpo[V(X_{j+1})|\sigalg_j] - V(X_j)\ri)\\ \non
	  \geq &\ V(X_n) - V(X_k)- \lf( \EE_\inpo[V(X_{k+1})|\sigalg_k] - V(X_k)\ri) + A(n-k-1)\\ 
	   \equiv&\ V(X_n) + A(n-k-1) - \EE_\inpo[V(X_{k+1})|\sigalg_k]. \label{eq:lst-k}
	\end{align}
	It follows that on $\{\lst = k\}$,
	\begin{align*}
	V(X_N)^r \leq \lf(|\mart _N - \mart_k| + \xi_k\ri)^r, \quad
	% \leq  \lf(|\mart _N - \mart_k| + \const_0\ri)^r
       % \end{align*}
       \text{and} \quad
    %    \begin{align*}
       A(N-k-1) \leq |\mart _N - \mart_k| + \xi_k,
        	\end{align*}
	where  $\xi_k =   \EE_\inpo[V(X_{k+1})|\sigalg_k]\indic{\{X_k \in \SC{D}\}}$.
	
\np	
    On $\{\lst= -\infty\},$  which corresponds to the case that the chain starting outside $\SC{D}$ never enters $\SC{D}$ by time $N$, we have 
  \begin{align*}
    V(X_N)^r \leq \lf(|\mart _N -\mart_0| + V(\inpo)\ri)^r, \quad \text{and} \quad  AN \leq |\mart _N-\mart_0| + V(\inpo).
     \end{align*}   	
	Thus for  $k\leq N-2$,
	\begin{align*}
	\EE_\inpo[V(X_N)^r1_{\{\lst=k\}}] \leq&\   \EE_\inpo\lf[\lf(|\mart _N - \mart_k| + \xi_k\ri)^r1_{\{\lst=k\}}\ri]\\
	             \leq &\ \EE_\inpo\lf[\lf(|\mart _N - \mart_k| + \xi_k\ri)^r1_{\{|\mart _N - \mart_k| + \xi_k \geq A(N-k-1)\}}\ri] \\
	             \leq &\ 2^{p-1} \lf(c_p(N-k)^{\f{p}{2}-1} \sum_{m=k}^{N-1} \EE_\inpo\lf[\vphi(X_m)\ri] + \EE_x\lf[\xi_k^p\ri]\ri)/(N-k-1)^{p-r}\\
	             \leq&\ \const_2 \lf((N-k)^{r-1-p/2}\sum_{m=k}^{N-1} \EE_\inpo\lf[\vphi(X_m)\ri] + (N-k)^{r-p}\ri), 
	\end{align*}
	where we used (a)\ \ref{t:maintheorem:bnds}, (b)\ Lemma \ref{l:aux1} along with the observation that 
	\begin{align*}
\EE_\inpo\lf[\abs{M_{n+1} - M_n}^p\ri] = \EE_\inpo\lf[\abs{V(X_{n+1}) - \EE[V(X_{n+1})|\sigalg_n]}^p\ri] \leq \EE_\inpo[\varphi(X_n)] 
	\end{align*}
	and (c)\ the fact that $\sup_{m\geq 2} m/(m-1) =2.$
	
	Similarly, on $\{\lst=-\infty\}$,		
	\begin{align*}
	\EE_\inpo[V(X_N)^r1_{\{\lst=-\infty\}}] \leq& \EE_\inpo\lf[\lf(|\mart _N| + V(\inpo)\ri)^r1_{\{|\mart _N| + V(\inpo) \geq AN\}}\ri] \\
	  \leq&\ 2^{p-1} \lf(c_pN^{\f{p}{2}-1} \sum_{m=0}^{N-1} \EE_\inpo\lf[\vphi(X_m)\ri] + V(\inpo)^p\ri)/N^{p-r}\\
	   \leq&\ 2^{p-1} \lf(c_pN^{r-1-\f{p}{2}} \sum_{m=0}^{N-1} \EE_\inpo\lf[\vphi(X_m)\ri] + V(\inpo)^pN^{r-p}\ri).
		\end{align*}
		Next, note that because of \ref{t:maintheorem:jumpbnd}
	\begin{align*}
	\EE_\inpo[V^p(X_N) | \sigalg_{N-1}]\indic{\{X_{N-1} \in \SC{D}\}} \leq 2^{p-1} \lf(\EE_\inpo[V(X_N) | \sigalg_{N-1}]^p\indic{\{X_{N-1} \in \SC{D}\}} + \sup_{x\in \SC{D}} \jumpbnd(x)\ri),
	\end{align*}	
	which by \ref{t:maintheorem:bnds} of course implies that for any $q \leq p$,
	$$\EE_\inpo[V(X_N)^q1_{\{\lst=N-1\}}] \leq\ \EE_\inpo[V^q(X_N)\indic{\{X_{N-1} \in \SC{D}\}}] \leq \const_3.$$
	Lastly,
	\begin{align*}
	\EE_\inpo[V(X_N)^r1_{\{\lst=N\}}] \leq& \  \EE_\inpo[V(X_N)^r1_{\{X_N \in \SC{D}\}}] \leq  \sup_{z\in \SC{D}} V^r(z), 
%	\EE_x[V(X_N)^r1_{\{\lst=N\}}] \leq&\ \EE[V^r(X_N)\indic{\{X_{N-1} \in \SC{D}\}}]
	\end{align*}
	Thus,
	\begin{align*}
	\EE_\inpo[V(X_N)^r] =&\ \sum_{k=0}^{N}\EE_\inpo[V(X_N)^r1_{\{\lst=k\}}] + \EE_\inpo V(X_N)^r1_{\{\lst=-\infty\}}]\\
	 \leq&\ \  \sum_{k=0}^{N-2}\EE_\inpo[V(X_N)^r1_{\{\zeta=k\}}] +\const_3+\sup_{z\in \SC{D}} V^r(z)+ \EE_\inpo[V(X_N)^r1_{\{\zeta=-\infty\}}]\\
	\leq &\ \const_4(1+ V^p(\inpo))\zeta(p-r) + \const_4\sum_{k=0}^{N-2}(N-k)^{r-1-p/2}\sum_{m=k}^{N-1} \EE_\inpo\lf[\vphi(X_m)\ri]\\
	\leq & \ \const_5(1+V^p(\inpo)) + \const_4\sum_{k=0}^{N-2}(N-k)^{r-1-p/2}\sum_{m=k}^{N-1} \EE_\inpo\lf[\const_1(\vep)+\vep V^r(X_m)\ri] \\
    \leq & \	\const_6(\vep)(1+V^p(\inpo)) + \const_4\vep\sum_{m=0}^{N-1} \beta^N_m\EE_\inpo\lf[V^r(X_m)\ri] ,
	\end{align*}
	where the choice of  $\vep$ will be specified shortly, and 
	$\beta^N_m = \sum_{k=0}^m (N-k)^{r-1-p/2}.$
	Iterating, we have
		\begin{equation}
	\label{it_eq}
	\begin{aligned}
	\EE_\inpo\bigl[ V(X_N)^r \bigr]	& \lq \const_6(\vep)(1+V^p(\inpo))\biggl(1+\const_4\vep \sum_{l_1=0}^{N-1}\beta^N_{l_1}+ (\const_4\vep)^2 \sum_{l_1=0}^{N-1}\beta^N_{l_1}\sum_{l_2=0}^{l_1-1}\beta_{l_2}^{l_1}+ \cdots\\
	& \qquad \cdots + (\const_4\vep)^{N-1}\beta^N_{N-1}\beta^{N-1}_{N-2}\hdots\beta^2_{1}\beta^1_{0}\biggr)(1+V^r(\inpo)). %\sum_{l_1=0}^{N-1}\beta^N_{l_1}\sum_{l_2=0}^{l_1-1}\beta_{l_2}^{l_1}\hdots\sum_{l_{N-1}=0}^{l_{N-2}-1}\beta_{l_{N-1}}^{l_{N-2}} \biggr).
	\end{aligned}
	\end{equation}
		Notice that for any $k>0$, since $r<p/2-1$,
	\[
	\sum_{l=0}^{k-1} \beta^k_l = \sum_{l=0}^{k-1}\sum_{j=0}^l (k-j)^{r-1-p/2} =   \sum_{j=0}^{k-1} (k-j)^{r-p/2} \lq \zeta(p/2-r).
	\]
	Choosing $\vep$ so that $\const_4 \vep \zeta(p/2-r) <1$, \eqref{it_eq} yields 
	\begin{align*}
	\EE_\inpo\bigl[ V(X_N)^r \bigr]	 \leq \frac{\const_6(\vep)(1+V^p(\inpo))(1+V^r(\inpo))}{1-\const_4 \vep \zeta(p/2-r)},
	\end{align*}
	%			and since by assumption \(r < p/2 - 1\), the right-hand side of \eqref{e:cell estimates} leads to
	%			\[
	%				\sup_{N\in\Nz} \EE\bigl[ X_N^r \bigr] \lq c_0 \exp\biggl( dC \sum_{\ell=1}^{+\infty} \ell^{r-p/2} \biggr) < +\infty,
	%			\]
	and the assertion follows.
	
	%{\em Second part:} We now consider the case, when $ p/2-1 \leq r < p\big(1-s/(p-2)\big) -1 $
		
	\end{proof}		
%%%%%%%%%%%%%%%%%%%%%

\vs{.2cm}
The next proposition helps to extend the above result from any $r<p/2-1$ to $\pow(s,p)$ as stipulated in Theorem \ref{t:maintheorem}. However it is also a stand-alone result that is applicable to certain models where Theorem \ref{t:maintheorem} is not directly applicable. These are cases where one directly does not have any good estimate of the conditional centered moment $\jdiff_n$ as required in Theorem \ref{t:maintheorem3}, but have suitable upper bounds for its $\|\cdot\|_\pwr$ norm. As a simple example, let $X_n$ be a stochastic process taking values in $[-\mfk{c}_0, \infty)$, whose temporal evolution is given by
$$X_{n+1} = \mfk{c}_1+X_n/2+ Y_n$$
where $\mfk{c}_0$ and $\mfk{c}_1$ are (real-valued) constants, and  $\{Y_n\}$ is an $\sigalg_n$-adapted martingale difference process (that is, $\EE(Y_{n+1}|\sigalg_n)=0$) and $\sup_n\EE(|Y_n|^p)<\infty$ for $p>2$. Then  Theorem \ref{t:maintheorem} is not  applicable, but the following proposition can be applied with $\pwr=1$ to $V(x) = x+\mfk{c}_0.$

\begin{proposition}
	\label{t:maintheorem3}
Let $\jdiff_n \equiv   \EE_\inpo\lf[|V(X_{n+1}) - \EE(V(X_{n+1})|\sigalg_n)|^p | \sigalg_n\ri]$ denote the centered conditional $p$-th moment of $V(X_{n+1})$ given $\sigalg_n$. Assume that \ref{t:maintheorem:driftbnd} and \ref{t:maintheorem:bnds} of Assumption \ref{cond_stability2} hold, and for $p>2$, some $\pwr \in [1, \infty]$ and some constant $0<\bar\cnst_\pwr(x)<\infty$,
\begin{align*}
\|\jdiff_n\|_\pwr = \EE_\inpo\lf[\jdiff_n^\theta\ri]^{1/\theta} \leq \bar\cnst_\theta(\inpo), \quad \text{ for all } n \geq 0.
\end{align*} 
Then   $\dst \cnst_r(\inpo) \doteq \sup_{n\in\N} \EE_\inpo \bigl[V(X_n)^r\bigr] < \infty$ for $0\leq r<  \bar\pow(\pwr,p),$
\begin{align*}
\bar\pow(\pwr,p) = 
\begin{cases}
p\lf(1-\f{1}{2\pwr}\ri)-1, & \quad \text{for } \theta \in \lf[1, \f{p}{2}\ri] \cup \lf(\f{p}{p-2}, \infty\ri] \text{ when } 2<p<4,\\
&\quad \text{for any }\ \theta \geq 1 \text{ when } p > 4; \\
p-2, & \quad \text{for } \theta \in \lf(\f{p}{2}, \f{p}{p-2}\ri] \text{ when } 2<p<4.
\end{cases}
\end{align*}
\end{proposition}	

Here  $\theta = \infty$ cooresponds to the case that  $\jdiff_n=\EE_\inpo \lf[|V(X_{n+1}) - \EE(V(X_{n+1})|\sigalg_n)|^p | \sigalg_n\ri] \leq \bar \cnst $ a.s, for some constant $ \bar \cnst>0$.

\begin{proof}[Proof of Proposition \ref{t:maintheorem3}] 
	The  constants appearing in various estimates below (besides the ones that appeared before) will be denoted by $\hat C_i$'s. They will not depend on $n$ but may depend on the parameters of the system and the initial position $\inpo$.
	
	Define $\mart_n$, $\lst$ and $\xi_k$ as in the proof of Proposition \ref{t:maintheorem2}. 	
Fix $N$,  $0\leq k \leq N$, and define $\hft \equiv \hft(N,k)$ by
\begin{align*}
\hft = \inf\{j\geq k: \mart_j - \mart_k + \xi_k \geq A(N-k-1)/2\}.
\end{align*}
Clearly, $\hft \leq N$. For $j>k$,  notice that on $\{\hft = j\}$, 
\begin{align*}
\mart_{j-1} - \mart_k + \xi_k \leq A(N-k-1)/2,
\end{align*}
and hence on $\{\lst = k\}\cap\{\hft = j\}$
\begin{align*}
\mart_N - \mart_{j-1} \geq A(N-k-1)/2 + V(X_N).
\end{align*}
It follows that for $j>k$
\begin{align*}
\EE_\inpo[V(X_N)\indic{\{\lst=k\}}\indic{\{\hft=j\}}] \leq \EE_\inpo\lf[|\mart_N - \mart_{j-1}|^r\indic{\{|\mart_N - \mart_{j-1}|>A(N-k-1)/2\}}\indic{\{|\mart_{j-1} - \mart_{k}|+\xi_k>A(j-k-2) \vee 0\}}\indic{\{\hft=j\}}\ri].
\end{align*}
Notice that $\SC{S}_j \equiv \EE_\inpo\lf[|\mart_N - \mart_{j-1}|^r\indic{|\mart_N - \mart_{j-1}|>A(N-k-1)/2} | \sigalg_j\ri] $ can be estimated by Lemma \ref{l:aux1} as
\begin{align*}
\SC{S}_j  \leq&\  (2/A(N-k-1))^{p-r}\EE_\inpo\lf[|\mart_N - \mart_{j-1}|^p | \sigalg_j\ri]  \\
\leq &\ \hat C_0 \lf[\EE_\inpo\lf[|\mart_N - \mart_{j}|^p | \sigalg_j\ri] +|\mart_j - \mart_{j-1}|^p \ri] / (N-k-1)^{p-r}\\
\leq &\ \hat C_0 \lf[\const_0 (N-j)^{\f{p}{2}-1}\EE_\inpo\lf[\sum_{l=j}^{N-1} \jdiff_l | \sigalg_j\ri] +|\mart_j - \mart_{j-1}|^p \ri] / (N-k-1)^{p-r}.
\end{align*}
Also, for $\hft=k$ by  Lemma  \ref{l:aux1}, 
\begin{align*}
	\EE_\inpo[V(X_N)\indic{\{\lst=k\}}\indic{\{\hft=k\}}] %\leq&\  \EE\lf[|\mart_N - \mart_{k}|^r\indic{||\mart_N - \mart_{k}|>A(N-k-1)}\indic{\{\hft=k\}}\ri]\\
	\leq&\ \EE_\inpo\lf[\indic{\{\hft=k\}} \EE_\inpo\lf[(|\mart_N - \mart_{k}| +\xi_k)^r\indic{||\mart_N - \mart_{k}| +\xi_k>A(N-k-1)} | \sigalg_k\ri]\ri]\\
	\leq &\ \hat C_1\EE_\inpo\lf[\indic{\{\hft=k\}}\lf((N-k-1)^{r-p/2-1}\sum_{l=k}^{N-1}\EE_\inpo\lf[\jdiff_l|\sigalg_k\ri] + (N-k-1)^{r-p}|\xi_k|^p \ri)\ri].
\end{align*}

Hence,
\begin{align}
\non
\EE_\inpo[V(X_N)\indic{\{\lst=k\}}] =&\ \sum_{j=k}^N\EE_\inpo[V(X_N)\indic{\{\lst=k\}}\indic{\{\hft=j\}}] \leq \ \hat C_1(N-k-1)^{r-p/2-1}\EE_\inpo\lf[\indic{\{\hft=k\}}\sum_{l=k}^{N-1}\EE_\inpo\lf[\jdiff_l|\sigalg_k\ri]\ri]\\
\non
& \hs{.5cm}+ \hat C_1 (N-k-1)^{r-p}\bar \cnst(\inpo) +  \sum_{j=k+1}^N\EE_\inpo\lf[\indic{\{\hft=j\}} \indic{\{|\mart_{j-1} - \mart_{k}|+\xi_k>A(j-k-2)\}} \SC{S}_j\ri] \\
\non
\leq &\  \hat C_2\Big[(N-k-1)^{r-\f{p}{2}-1} \sum_{j=k}^{N}\sum_{l=j}^{N-1}\EE_\inpo\lf[\jdiff_l\indic{\{\hft=j\}}\ri] + (N-k-1)^{p-r} \\
& \hs{.3cm}+  (N-k-1)^{r-p}\sum_{j=k+1}^N|\mart_j - \mart_{j-1}|^p\indic{\{|\mart_{j-1} - \mart_{k}|+\xi_k>A(j-k-2)\vee 0\}}\Big].
\label{eq:V-est-0}
\end{align}
We next estimate the above two terms separately, and for that we need the following bound which is an immediate consequence of Doob's maximal inequality and assumption:
\begin{align}
	\PP_\inpo\lf(\max_{k\leq j \leq l}|\mart_j-\mart_k| +\xi_k > \Upsilon\ri) \leq&\ \EE_\inpo\lf[\max_{k\leq j \leq l}|\mart_j-\mart_k| +\xi_k\ri]^p/\Upsilon^p \leq \hat C_3\lf((l-k)^{p/2}+\bar \cnst_0(x)\ri)/\Upsilon^p \label{prob-est}
	\end{align}  

Now notice that
\begin{align}
\non
\sum_{j=k}^{N-1}\sum_{l=j}^{N-1}\EE_\inpo\lf[\jdiff_l\indic{\{\hft=j\}}\ri] =&\ \sum_{l=k}^{N-1}\sum_{j=k}^l\EE_\inpo\lf[\jdiff_l\indic{\{\hft=j\}}\ri] = \sum_{l=k}^{N-1}\EE_\inpo\lf[\jdiff_l\indic{\{\hft \leq l\}}\ri]
\leq \ \sum_{l=k}^{N-1}\|\jdiff_l\|_\pwr\PP_\inpo(\hft \leq l)^{1/\pwr^*} \\
\non
\leq &\ \bar\cnst_\theta(x) \sum_{l=k}^{N-1}\PP_\inpo\lf(\max_{k\leq j \leq l}\lf(|\mart_j-\mart_k| +\xi_k\ri) > A(N-k-1)/2\ri)^{1/\pwr^*} \\
\non
\leq & \lf(\f{2}{A}\ri)^{p/\pwr^*}\bar\cnst_\theta(x) \sum_{l=k}^{N-1}\lf(\f{\hat C_3\lf((l-k)^{p/2}+\bar\cnst_0(x)\ri)}{(N-k-1)^p}\ri)^{1/\pwr^*}\\
%\leq &\ \hat C_1 (N-k-1)^{-p/\pwr^*} \sum_{l=k}^N \|\jdiff_l\|_\pwr\lf((l-k)^{\f{p}{2}-1}\sum_{j=k}^l\EE[\jdiff_j] + \EE[\xi_k^p]\ri)^{1/\pwr^*}\\
\leq &\ \hat C_4 (N-k-1)^{-\f{p}{2\pwr^*}+1} =  \hat C_4 (N-k-1)^{-\f{p}{2}\lf(1-\f{1}{\pwr}\ri)+1},
\label{eq:V-est-1}
\end{align}
where $\f{1}{\pwr}+\f{1}{\pwr^*}=1$.

Next notice that the term
$$\SC{A} \equiv \sum_{j=k+1}^N\EE_\inpo\lf[|\mart_j - \mart_{j-1}|^p\indic{\{|\mart_{j-1} - \mart_{k}|+\xi_k>A(j-k-2)\}}\ri]$$
can be estimated as
\begin{align} 
\non
\SC{A} \leq &\ \|\jdiff_{k+1}\|_\theta+\|\jdiff_{k+2}\|_\theta +  \sum_{j=k+3}^N\EE_\inpo\lf[\jdiff_{j-1}\indic{\{|\mart_{j-1} - \mart_{k}|+\xi_k>A(j-k-2)\}}\ri]\\
\non
 \leq &\ 2\bar\cnst_\theta(x)+ \sum_{j=k+3}^N\|\jdiff_{j-1}\|_\pwr\PP_\inpo\lf(|\mart_{j-1} - \mart_{k}|+\xi_k>A(j-k-2)\ri)^{1/\pwr^*}\\
 \non
\leq &\   \bar\cnst_\theta(x)\lf[2+A^{-p/\theta^*}\sum_{j=k+3}^N \lf(\hat C_3\lf((j-k-1)^{p/2}+\bar \cnst_0(x)\ri)/(j-k-2)^p\ri)^{1/\pwr^*}\ri], \\
\non
\leq&\ \hat C_5\lf[1+\sum_{j=k+3}^N 1/(j-k-2)^{p/2\theta^*}\ri]\\
\leq&\ 
\begin{cases}
 \hat C_6,& \quad \text{ if } p/2\pwr^* = p(1-1/\theta)/2 >1,\\
\hat  C_7(N-k-1),& \quad \text{ otherwise},
\end{cases}
\label{eq:A-est}
\end{align}
where the third inequality is by  \eqref{prob-est}.
We now consider some cases.

{\em Case 1: $\theta \leq p/2$:} Suppose that  $r < p\lf(1-\f{1}{2\theta}\ri)-1$. Notice in this case this implies that $p-r-1 > \f{p}{2\theta} \geq 1.$
It follows from \eqref{eq:V-est-0}, \eqref{eq:V-est-1} and \eqref{eq:A-est} (second case) that % for $r-\f{p}{2}\lf(1+\f{1}{\pwr^*}\ri)<-1$,
\begin{align*}
\EE_\inpo[V^r(X_N)] =&\ \sum_{k=0}^N\EE_\inpo[V^r(X_N)\indic{\{\lst=k\}}] \leq \sum_{k=0}^{N-2}\EE_\inpo[V^r(X_N)\indic{\{\lst=k\}}] +\EE_\inpo[V^r(X_N)\indic{\{X_{N-1} \in \SC{D}\}}]+  \sup_{x \in \SC{D}} V^r(x)\\
\leq &\  \hat C_8\lf[\sum_{k=0}^{N-2} (N-k-1)^{r-\f{p}{2}-1-\f{p}{2}\lf(1-\f{1}{\pwr}\ri)+1} +\sum_{k=0}^{N-2} (N-k-1)^{r-p}+ \sum_{k=0}^{N-2} (N-k-1)^{r-p+1}\ri]\\
&\ + \const_3+\sup_{x \in \SC{D}} V^r(x)\\
 = &\ \hat C_8\lf( \zeta\lf( p\lf(1-\f{1}{2\theta}\ri)-r\ri)+ \zeta(p-r)+\zeta(p-r-1)\ri)+  \const_3+\sup_{x \in \SC{D}} V^r(x).
 \end{align*}
{\em Case 2: $\theta > p/2$, and $p \geq 4$:} Suppose that  $r < p\lf(1-\f{1}{2\theta}\ri)-1$. Notice that $\theta > p/2$, and $p \geq 4$ imply that $p/2\theta^* = p(1-1/\pwr)/2 >1$. Like the previous case, it again follows from \eqref{eq:V-est-0}, \eqref{eq:V-est-1} and \eqref{eq:A-est} (first case)
  \begin{align*}
\EE_\inpo[V^r(X_N)]  \leq& \ \sum_{k=0}^{N-2}\EE_\inpo[V^r(X_N)\indic{\{\lst=k\}}] +\EE_\inpo[V^r(X_N)\indic{\{X_{N-1} \in \SC{D}\}}]+  \sup_{x \in \SC{D}} V^r(x)\\
 \leq &\ \hat C_9\lf(\sum_{k=0}^{N-2} (N-k)^{r- p\lf(1-\f{1}{2\theta}\ri)} + \sum_{k=0}^{N-1} (N-k-1)^{r-p}\ri)+  \const_3+\sup_{x \in \SC{D}} V^r(x)
 \leq \ \hat C_{10}.
 \end{align*}
 The other cases in the assertion follow similarly once we observe that $\theta > p/(p-2) \LRT p/2\theta^* >1$ and for $2<p<4$, $p/2 < p/(p-2)$.

\end{proof}

\subsection{Ergodicity of Markov processes}

Theorem \ref{t:maintheorem} leads to the following result on Harris ergodicity of Markov processes.

\begin{definition}
A function $V:\metsp \rt [0,\infty)$ is {\em inf-compact} if the level sets, $\SC{K}_m = \{x:V(x) \leq m\}$ are compact for all $m\geq 0$.
\end{definition}
Note that an inf-compact function $V$ is lower-semicontinuous.	 

\begin{theorem}\label{MP-stab}
Let $\{X_n\}$ be a Markov process taking values in a locally compact separable space $\metsp$ with transition kernel $\tk$. Suppose for an inf-compact function $V:\metsp \rt [0,\infty)$, the following conditions hold:
\begin{enumerate}[label={\rm (\ref{MP-stab}-\alph*)}, widest=b, leftmargin=*, align=left, nolistsep]
		\item \label{MP-stab-i} for all $n\in N$, 
		$$\tk V(x) - V(x) \leq  -A, \quad \text{on } \ \{x \notin  \SC{D}\};$$
		\item \label{MP-stab-ii} for some $p > 2$
		$$\tk|V(\cdot) - \tk V(x)|^p(x) = \int |V(y) - \tk V(x)|^p\tk(x,dy) \leq \jumpbnd(x),$$
		where $\jumpbnd: \metsp \rt [0,\infty]$ satisfies $\jumpbnd(x) \leq \const_{\jumpbnd}(1+V^s(x))$ for some $s<p/2-1$ and some constant $\const_{\jumpbnd}>0$. This is of course same as  requiring  $ \EE\lf[\big|V(X_{n+1}) - \tk V(X_{n})\big|^p \Big| \sigalg_n\ri] \lq \jumpbnd(X_n).$
	%	\item \label{cond:growth} for some constant $C_r\geq 0$
	%	$$r(x)\lq C_r(1+V(x)), \quad \mbox{ for all}\ \ x\in \SC{S}.$$
		\item  \label{MP-stab-iii} $\sup_{x\in  \SC{D}}V(x)<\infty,$ and $\sup_{x\in  \SC{D}}\tk V(x)<\infty,$
		%for some constant $B$,  
%		$$\EE\lf[\lf(\tk V(X_{n})\ri)^p\indic{\{X_n \in  \SC{D}\}} \ri] < B.$$
%		\item $\dst M\equiv \sup_{x\in K}\EE_x\lf(V(X_1)\indic{\{X_1\notin K\}}\ri) <\infty.$
%		\item $\dst\bar{\L}\equiv \sum_{n\geq 0}n\l_n <\infty.$
\end{enumerate}
Also, suppose that 
\begin{enumerate}[label={\rm (\ref{MP-stab}-d)}, widest=b, leftmargin=*, align=left, nolistsep]
\item \label{MP-stab-iv} $\tk$ is weak Feller, $\psi$-irreducible, and admits a density $\tkd$ with respect to some Radon  measure $\mu$, that is, $\tk(x,dy)=\tkd(x,y)\mu(dy)$, and that for every compact set $\SC{K}$, there exists a constant $\mfk{c}_{\SC{K},0}$
 such that $$\sup_{y \in \SC{K}}\tkd(x,y) \leq \mfk{c}_{\SC{K},0}\lf(1+V^r(x)\ri).$$ 
 %for all $y \in \metsp$
\end{enumerate}
Then
\begin{enumerate}[label={\rm(\roman*)}]
\item Under \ref{MP-stab-i} - \ref{MP-stab-iii}, $\sup_n\EE_\inpo(V^r(X_n)) \equiv \sup_n \tk^nV^r(\inpo)  < \infty$ for any $0\leq r<\pow(s,p)$, where $\pow(s,p)$ is as in Theorem \ref{t:maintheorem}.
\item Under additional assumption of \ref{MP-stab-iv}, $\{X_n\}$ is positive Harris recurrent (PHR) and aperiodic with a unique invariant distribution $\pi$, and for any $x_0$ and $r \in (0,\pow(s,p))$
\begin{align}\label{eq:conv-1}
\int (V^{r}+1)d|\tk^n(\inpo,\cdot) -\pi| \rt 0 \quad \text{as } n \rt \infty;
\end{align}
or equivalently,
\begin{align}\label{eq:conv-2}
\|\tk^n(\inpo,\cdot) -\pi\|_{V^r+1} \doteq \sup_{f:|f| \leq V^r+1}|\tk^nf(\inpo) -\pi(f)| \rt 0, \quad \text{as } n \rt \infty.
\end{align}
\end{enumerate}
\end{theorem}

\begin{proof}

(i) follows from the Theorem \ref{t:maintheorem}. 
Since $V$ is inf-compact, it follows from (i) that for every $x_0$, $\{\tk^n(x_0,\cdot)\}$ is tight, and let $\pi$ be one of its limit point. Since $\tk$ is weak Feller, by the Krylov-Bogolyubov theorem \cite[Theorem 7.1]{ref:daP-06}, \(\pi\) is invariant for $\tk$, and uniqueness of $\pi$ follows from the assumption of $\psi$-irreducibility  \cite[Proposition 4.2.2]{HL03} . Hence, for every $x_0$, $\tk^n(x_0,\cdot) \RT \pi$ (along the full sequence) as $n \rt \infty$.

For (ii) we start by establishing the following claim.

\np
{\em Claim:} Suppose that $f \leq V^r + 1$ for some $r \in (0,\pow(s,p))$. Then $\tk^nf(x_0) \rt \pi(f)$ as $n \rt \infty$ for any $x_0 \in \metsp$.

\np
Since $V$ is lower semi-continuous we have by  (generalized) Fatou's lemma,
\begin{align*}
\pi(V^r) \leq \liminf_{n\rt \infty} \tk^nV^r(x_0) \leq \cnst_r(x_0)
\end{align*}
for any $r \in (0,\pow(s,p))$. Now let $f \leq V^r+1$ for some $r \in (0,\pow(s,p))$ and fix $\vep>0$.

Since $\{\tk^n(x_0,\cdot)\}$ is tight, for a given $\tilde\vep>0$, there exists a compact set $\SC{K}$ (which depends on $x_0$ and which we take of the form $ \SC{K}_{m} = \{x:V(x) \leq m\}$ for sufficiently large $m$) such that
$$\sup_n \tk^n(x_0,\SC{K}^c) \leq \tilde\vep, \quad \text{ and } \quad \pi(\SC{K}^c) \leq \tilde\vep.$$
Now by H\"older's inequality
\begin{align}
\non
\tk^nf1_{\SC{K}^c}(x_0) =&\ \int f(y)1_{\SC{K}^c}(y)\tk^n(x_0,dy)  \leq  \int (V^r(y)+1)1_{\SC{K}^c}(y)\tk^n(x_0,dy) \\
\non
\leq&\  \lf(\int V^{r'}(y)\tk^n(x_0,dy)\ri)^{r/r'}\lf( \int \indic{\SC{K}^c}(y)\tk^n(x_0,dy)\ri)^{1 - r/r'} + \tk^n(x_0, \SC{K}^c ),\\
\label{eq:Kout}
\leq&\ \cnst_{r'}^{r/r'}(x_0) \tilde\vep^{1 - r/r'} + \tilde\vep
\end{align}
for some $r<r'<\pow(s,p)$.
Similarly, $\pi(f\indic{\SC{K}^c}) \leq \cnst_{r'}^{r/r'}(x) \tilde\vep^{1 - r/r'} +\tilde\vep$.

Since $f\indic{\SC{K}} \in L^1(\mu)$, there exist $\{h_m\} \subset C_c(\metsp,\R)$  such that $h_m \rt f1_{\SC{K}}$ in $L^1(\mu)$ as $m\rt \infty$, and  $\sup_x|h_m(x)| \leq \sup_{x\in \SC{K}}|f(x)|$ for $m\geq 1$. In fact, we can choose $\{h_m\}$ such that $supp(h_m) \subset \SC{K}' \supset \SC{K}.$ 
for some compact set $\SC{K}'$.

Observe that for $x\in \metsp$ $y \in \SC{K}'$
\begin{align*}
q^n(x,y) =&\  \int q^{n-1}(x,z)q(z,y) d\mu(z) \leq   \int q^{n-1}(x,z)\mfk{c}_{\SC{K'},0}\lf(1+V^r(z)\ri) d\mu(z) \\
\leq &\ \mfk{c}_{\SC{K'},0}\lf(1+ \EE_x(V^r(X_{n-1}))\ri) \leq \mfk{c}_{\SC{K'},0}\lf(1+\cnst_r(x)\ri) \equiv \const_{\SC{K}'}(x).
\end{align*}

Hence
\begin{align}
\non
\sup_n|\tk^n f\indic{\SC{K}}(x_0) - \tk^n h_m|  \leq&\ \int_{\SC{K}'} |f(y)\indic{\SC{K}}(y) - h_m(y)| \tkd^n(x_0,y) d\mu(y)\\
 \leq &\ \const_{\SC{K}'}(x_0) \|f1_{\SC{K}} - h_m\|_1.
\label{est-fK}
\end{align}
Next, notice that $\pi$ is absolutely continuous with $\mu$. Indeed, if $\mu(A) = 0$, then $\tk(x,A)=0$, and hence $\pi(A) = \int \pi(dx) \tk(x,A) =0$.
Let $g = d\pi/d\mu$. For any $M>0$,
\begin{align}
\non
|\pi(h_m) - \pi(f1_{\SC{K}})| \leq &\  M \int |h_m-f\indic{\SC{K}}|\indic{\{g\leq M\}} d\mu+ \int |h_m-f\indic{\SC{K}}|g\indic{\{g\geq M\}} d\mu\\
\label{est-pifK}
\leq & M\|h_m - f\indic{\SC{K}}\|_1+ 2\sup_{x \in \SC{K}}|f(x)| \int g \indic{\{g\geq M\}}d\mu.
\end{align}
Write
\begin{align}
\non
\tk^n f(x_0)  - \pi(f) =&\ \lf(\tk^n f1_{\SC{K}}(x_0)-\tk^n h_m(x_0)\ri) + \lf(\tk^nh_m(x_0) - \pi(h_m) \ri)+ \lf(\pi(h_m) - \pi(f1_{\SC{K}}(x_0))\ri)\\
\label{eq:split-pi}
& \ +\tk^n f1_{\SC{K}^c}(x_0)-\pi(f1_{\SC{K}^c}(x_0)),
\end{align}
and choose $\SC{K}$ such that \eqref{eq:Kout} holds for $\tilde \vep$ where $\tilde \vep$ is chosen such that $\cnst_{r'}^{r/r'}(x_0) \tilde\vep^{1 - r/r'} + \tilde\vep \leq \vep/10$. 
Since $\int gd\mu = 1$, choose sufficiently large $M$ such that $\int g \indic{\{g\geq M\}}d\mu \leq \vep / (20\sup_{x \in \SC{K}}|f(x)|)$, then a sufficiently large $m$ such that
$$\|f1_{\SC{K}} - h_m\|_1 \leq (\vep/5 \const_{\SC{K}'}(x_0))\wedge (\vep/10M).$$
Finally, since $\tk^n(x_0,\cdot) \RT \pi$, and $h_m \in C_c(\metsp,\R)$,    we have $ \lf(\tk^nh_m(x_0) - \pi(h_m) \ri) \rt 0$ as $n\rt \infty$. Hence, we can choose a sufficiently large $n$ such that $| \tk^nh_m(x_0) - \pi(h_m)| \leq \vep/5$, and thus from  \eqref{eq:Kout}, \eqref{est-fK}, \eqref{est-pifK}
and \eqref{eq:split-pi},
\begin{align*}
|\tk^n f(x_0)  - \pi(f)| \leq & \vep.
\end{align*}

\np
This proves the claim, which in particular says that for any $x_0 \in \metsp$ and any Borel set $A$, $\tk^n(x,A) \stackrel{n\rt \infty}\rt \pi(A)$.  
%
%From there to go to the more stronger convergence in the assertion (ii), we need to use a woderful result from \cite{HL03}, which says for positive Harris chains setwise convergence of $\tk^n(x,\cdot)$ to $\pi$ is equivalent to convergence in total variation norm.
%
By \cite[Theorem 4.3.4]{HL03} (also see \cite{HL01}), $\{X_n\}$ is aperiodic and PHR, and by the same result this implies $\|\tk^n(x,\cdot) - \pi\|_{TV} \rt 0.$ The equivalence of the setwise convergence of $\tk^n(x,\cdot)$ and convergence in total-variation norm is a unique feature of PHR chains. Now note that by H\"older's inequality for some $r' \in (r, \pow(s,p))$
\begin{align*}
%\|\tk^n(x,\cdot) -\pi\|_{V^r+1} \leq&\ \sup_{f \leq V+1} \int f(x) d|\tk^n(x,\cdot)- \pi|(dy) \leq \   
\int (V^r(y) +1)d|\tk^n(x,\cdot)- \pi|(y) \leq&\ \lf(\int V^{r'}(y)(\tk^n(x,dy)+ \pi(dy))\ri)^{r/r'}\|\tk^n(x,\cdot)- \pi\|_{TV}^{1-r/r'}\\
& \hs{.4cm}+\|\tk^n(x,\cdot)- \pi\|_{TV}\\
\leq & \ 2\cnst_{r'}(x)^{r/r'}\|\tk^n(x,\cdot)- \pi\|_{TV}^{1-r/r'}+\|\tk^n(x,\cdot)- \pi\|_{TV} \stackrel{n\rt \infty}\rt 0.
\end{align*}
The equivalence of \eqref{eq:conv-1} and \eqref{eq:conv-2} follows from  Lemma \ref{lem:eq-g-norm} below.

\end{proof}

\begin{lemma}\label{lem:eq-g-norm}
Let $\nu$ be a signed measure on a complete separable metric space $\metsp$. Suppose that $g:\metsp \rt [0,\infty)$ is a measurable function such that $|\nu |(g) = \int gd|\nu| < \infty$. Then 
$$ \f{1}{2} |\nu |(g) \leq\|\nu\|_{g} \leq |\nu |(g),$$
where recall $\|\nu\|_{g}  = \sup_{f:|f| \leq g} |\nu(f)|$
\end{lemma}

\begin{proof}
The last inequality is trivial as for any measurable $f$ with $|f| \leq g$, $|\nu(f)| \leq |\nu|(|f|) \leq |\nu |(g)$. For the first inequality, let $\metsp = \SC{Y} \cup \SC{N}$ be the Hahn decomposition for $\nu$ (in particular, $\SC{Y} \cap \SC{N} =\emptyset$) , with the corresponding Jordan decomposition $\nu = \nu^+-\nu^-$ (i.e., supp$(\nu^+) \subset \SC{Y}$ and supp$(\nu^-) \subset \SC{N})$.
Choose $f = g\indic{\SC{Y}}$. Then
 \begin{align*}
 \|\nu\|_{g} \geq |\nu(g\indic{\SC{Y}})| =| \nu^+(g\indic{\SC{Y}}) -  \nu^-(g\indic{\SC{Y}})| =  \nu^+(g\indic{\SC{Y}}) = \nu^+(g), 
 \end{align*}
 where the last equality is because supp$(\nu^+) \subset \SC{Y}$. Similarly, choosing  $f = g\indic{\SC{N}}$, we have $ \|\nu\|_{g} \geq \nu^-(g) $, whence it follows that $ 2\|\nu\|_{g} \geq |\nu|(g).$

\end{proof}

%% file: applications.tex
\section{Applications}\label{sec-app}
		
This sections is devoted to understanding stability of a broad class of  multiplicative systems through application of the previous theorems. 
 					
			\subsection{Discrete time switching systems}
			Let $\H$ be a Hilbert space and $\metsp$  a Polish space. 
		Suppose there exists a sequence of measurable maps $P_n:\H \times \metsp \times \metsp\rt[0, 1]$ such that for each $x\in \H$, the function $P_n(x, \cdot, \cdot)$ is a transition probability kernel.	Consider a discrete-time $\SC{F}_n$-adapted process $\{Z_n\} \equiv \{(X_n, Y_n)\}$ taking values in $\H\times \metsp$, whose dynamics is defined by the following rule: given the state $(X_n, Y_n) = (x_n, y_n)$, 
			\begin{enumerate}[label={\rm (SS-\arabic*)}, leftmargin=*, align=left, start=1]
				\item \label{hyb-1} first, $Y_{n+1}$ is selected randomly according to the (possibly) time-inhomogenous  transition probability distribution $P_n(x_n, y_n, \cdot) \Let P_{n, x_n}(y_n, \cdot)$, 
				\item \label{hyb-2} next given $Y_{n+1} = y_{n+1}$,
				   $$X_{n+1} = H_n(x_n, y_{n+1}, \xi_{n+1}),$$
				   where  $\{\xi_k:k=1,\hdots\}$ is a sequence of independent random variables taking values in a Banach space $\B$,  $\xi_{n+1}$ is independent of $\s\{\SC{F}_n, Y_{n+1}\}$ and $H_n: \H\times \metsp\times \B \rt \H$.
				 \end{enumerate}
In general $\{(X_n,Y_n)\}$ is a (possibly) time-inhomogeneous Markov process but clearly, neither $\{X_n\}$ nor $\{Y_n\}$ is Markovian on its own. 	The stochastic system $\{(X_{n},Y_n)\}$ is known as a {\em discrete-time switching system} or a {\em stochastic hybrid system} (and sometimes also known as 	\emph{iterated function system with place dependent probabilities} \cite{ref:BarDemEltGer-88}). Stochastic hybrid systems are extensively used to model practical phenomena where system parameters are subject to sudden changes. These systems have found widespread applications in various disciplines including synthesis of fractals, modeling of biological networks,  \cite{ref:LasMac-94}, target tracking \cite{MM90}, communication networks \cite{JPH05}, control theory \cite{ref:ChaPal-11, ref:ChaCinLyg-11, ref:CosFraMar-05} - to name a few.  There is a considerable literature addressing classical weak stability questions concerning the existence and uniqueness of invariant measures of iterated function systems, see e.g., \cite{ref:Pei-93, ref:LasYor-94, ref:Sza-03, ref:DiaFre-99, ref:JarTwe-01} and the references therein. Comprehensive sources studying various properties of these systems including results on stability  in both continuous and discrete time can be found in \cite{MaYu06, YiZh10} (also see the references therein). In most of these works, $Y_n$ is often assumed to be a stand-alone finite or countable state-space Markov chains.

%		
%		Iterated function systems are important objects of study in control theory, where they are known by the name discrete-time stochastic hybrid systems \cite{ref:ChaPal-11, ref:ChaCinLyg-11, ref:CosFraMar-05}.  The arguments in these articles predominantly revolve around average contractivity conditions of the iterated function system, and continuity of the probability transitions. Stronger stability notions such as existence of moments of sufficiently high order mostly involve Foster-Lyapunov drift conditions, which in turn work best under the average contractivity assumption. Although there have been efforts to relax average contractivity conditions in conjunction with Foster-Lyapunov drift conditions, see e.g., \cite{ref:DouForMouSou-04}, generally the assertions consist of sub-geometric rates of convergence of Markov processes to their invariant measures; moreover, such techniques do not extend directly to moment bounds. Furthermore, in real-world control applications the average contractivity property generally translates to requiring unbounded control actions, which is hardly ever possible to guarantee. 
%		
	
\np		
We consider a broad class of coupled switching or hybrid systems whose dynamics is described by \ref{hyb-1} and \ref{hyb-2} with $H_n$  of the form
$$H_n(x,y,z) =  L_n(x,y)+ F_n(x,y) + G_n(x,y,z),$$
where $L_n, F_n:\H \times \metsp \rt \H$ and $G_n:\H \times \metsp\times \B \rt \H$. 
In other words, $\{X_n\}$ satisfies
\begin{align}
	\label{eq:ss-2}
X_{n+1} = L_n(X_n, Y_{n+1})+ F_n(X_n, Y_{n+1}) + G_n(X_n,Y_{n+1},\xi_{n+1})
\end{align}
where the $\xi_n$ are $\B$-valued random variables.  \eqref{eq:ss-2}, for example, 
 includes multiplicative systems of the form
 \begin{align*}
%	\label{eq:ss-2}
X_{n+1} = X_n+ F_n(X_n, Y_{n+1}) + G^0_n(X_n,Y_{n+1})\xi_{n+1}.	
\end{align*}
 
  We will make the following assumptions on the above system.

\begin{assumption} \label{assum-SS2}	

\ \

			\begin{enumerate}[label={\rm (SS-\arabic*)}, leftmargin=*, align=left, start=7]
				%\item \label{H-SS-L} $L_n: \H\times\metsp \rt \H$ satisfies $\|L_n(x,y)\| \leq \|x\| + \mfk{c}_0$ for some constant $\mfk{c}_0$. (for example, $L_n(x,y) \equiv L_n(x) = U_nx$, where $U_n$ is a unitary operator.)
				
				\item \label{H-SS-drift} For $\|x\| >B$, and any $y \in \metsp$,
				$$ P_{n,x}\<F_n(x,\cdot), L_n(x,\cdot)\>(y) = \int \<F_n(x,y'), L_n(x,y')\> P_{n,x}(y,dy') \leq - \mfk{m}_0\|x\|^{-(1+\g)}, $$
				% \item  $\|L_n(x,y)\| \leq \mfk{m}_{L}(y)\|x\|,$
%				\end{itemize}
%				\vs{.2cm}
			 for some constants $\mfk{m}_0$ and exponent $\g\geq 0$.  \vs{.1cm}
				
			\item \label{H-SS-FG} The following growth conditions hold:
			\begin{itemize}
			\item $ \|L_n(x,y)\| \leq \mfk{m}_{L,1}(y)\|x\|+\mfk{m}_{L,2}(y)$% for $\|x\| >B$,\  $ \sup_{\|x\| \leq B } \|L_n(x,y)\| \leq \mfk{m}_{L,2}(y),$
			 and $\dst \|\bar L_n(x,y)\| \leq \mfk{m}_{\bar L}(y)(1+\|x\|)^{l_1},$  where\\
		       % \begin{align*}
		%	\label{barF}
		$	\bar L_n(x,y) = L_n(x,y) - P_{n,x}L_n(x,\cdot)(y).$
		%	\end{align*}
		\vs{.1cm}
			
			 \item $\dst \|F_n(x,y)\| \leq \mfk{m}_F(y)(1+\|x\|)^{f_0},$ $\bar F_n(x,y) \leq \mfk{m}_{\bar F}(y)(1+\|x\|)^{f_1}$,  \\$ \|G_n(x,y,z)\| \leq \mfk{m}_G(y)(1+\|x\|)^{g_0}\Psi(z),$  where $\Psi:\B \rt [0,\infty)$ and
		%	\begin{align*}
		%	\label{barF}
		$	\bar F_n(x,y) = F_n(x,y) - P_{n,x}F_n(x,\cdot)(y).$
		%	\end{align*}	
		\vs{.1cm}
			
			\item  For any $p>0$, the constants $\bar{\mfk{m}}_{F,p}, \bar{\mfk{m}}_{\bar F,p}, \bar{\mfk{m}}_{G,p}, \bar{\mfk{m}}_{L,1,p}, \bar{\mfk{m}}_{L,2,p}$ and $\bar{\mfk{m}}_{\bar L,p}$ are finite, and $\bar{\mfk{m}}_{L,1,2} \leq 1$ ,where the above constants are defined as
			\begin{align}\label{eq:bar-const}
			\bar{\mfk{m}}_{\chi,p} \doteq \sup_{n,x,z} \int  \mfk{m}^p_\chi(y)P_{n,x}(z,dy), \quad \chi=F, \bar F, G, \{L,1\},  \{L,2\}, \bar L.
			\end{align}
			
			  \end{itemize}
			  
			  \item \label{H-SS-exp} The exponents satisfy:
			  \begin{itemize}
			  \item (a) 	  $f_0 < (1+\g)/2$, or (b)\ $f_0 = (1+\g)/2$ and $\bar{\mfk{m}}_{F,2} \leq 2\mfk{m}_0$;
			  \item  $\ g_0<\g\wedge 1/2$, and $l_1\vee f_1 <1/2$. 	
			  \end{itemize}		  
			 \item \label{H-SS-noise} The $\xi_n$ are independent $\B$-valued random variables with distribution $\nu_{n}$; for each $n$, $\xi_{n+1}$ is independent of $\s\{\SC{F}_n, Y_{n+1}\}$, and for any $p>0$, $m_*^p = \sup_n\EE(\Psi(\xi_n)^p) < \infty$
			\end{enumerate}
\end{assumption}				
			
\begin{proposition}\label{prop:switch}
%Suppose that the noise process $\{\xi_n\}$ is zero-mean $\B$-valued random variables (that is, $\EE(\xi_n)=0$ for all $n$). 
Under Assumption \ref{assum-SS2}, $\sup_n \EE_{x_0}\|X_n\|^m < \infty.$ for any $m>0$ and $x_0 \in \H$. If the functions $G_n$ are centered with respect to the variable $z$ in the sense that $\dst \hat G_n(x,y) \doteq \ \int_\B G_n(x,y,z) \nu_{n+1}(dz) = 0$ for all $n\geq 1$,  $x \in \H$ and $y\in \metsp$, then we only need $g_0<1/2$ instead of $\ g_0<\g\wedge 1/2$ in \ref{H-SS-exp}   for the above assertion  to be true.
\end{proposition}	
			
\np			
\begin{remark}\label{rem-swsys}{\rm
A few comments are in order.
\begin{itemize} 
\item Because of the growth assumption on $G_n$ in \ref{H-SS-FG} and the condition \ref{H-SS-noise}, for each $n,x$ and $y$, the function $z \rt G_n(x,y,z)$ is Bochner integrable, and hence  $\hat G_n(x,y) \doteq \int_\B G_n(x,y,z) \nu_{n+1}(dz)$ is well defined (the integral is defined in Bochner sense).\vs{.15cm} 

\item One scenario where the functions $G_n$  are centered (with respect to the variable $z$) occurs when considering multiplicative stochastic system driven by zero-mean random variables. Specifically, in such models  the $G_n$ are of the form 
$G_n(x,y,z) = G^0_n(x,y)z$ and the $\xi_n$ are mean zero-random variables. Also notice for these models, $\Psi(z) = \|z\|_\B.$\vs{.15cm} 

\item Suppose that the $G_n$ are not centered in the variable $z$. If $\g<1/2$, \ref{H-SS-exp} requires that the growth exponent of $G_n$, $g_0<\g$. However, this could be extended to the boundary case of $g_0 = \g$ (when $\g<1/2$) provided the averaged growth constants $\bar{\mfk{m}}_{\chi,p}$ (c.f. \eqref{eq:bar-const})  meet certain conditions. If $g_0=\g$ and $f_0<(1+\g)/2$, then the assertion of Proposition \ref{prop:switch} is true provided $\lf(\bar{\mfk{m}}_{G,2}m^2_*\ri)^{1/2}<\mfk{m}_0$. If $g_0=\g$ and $f_0 = (1+\g)/2$, then the same assertion  holds provided $\lf(\bar{\mfk{m}}_{G,2}m^2_*\ri)^{1/2}+ \bar{\mfk{m}}_{F,2}/2<\mfk{m}_0$. \vs{.15cm} 

\item Condition	\ref{H-SS-drift} is implied by the simpler condition: 
	$$\<F_n(x,y),L_n(x,y)\> \leq - \mfk{m}_0\|x\|^{1+\g}, \quad \|x\|>B, \ \forall y.$$
Similarly, for many models a stronger (but easier to check) form of the condition \ref{H-SS-FG} , where the `constants' $\mfk{m}_{\chi}$ (for $ \chi=F, \bar F, G, \{L,1\},  \{L,2\}, \bar L$) do not depend on $y$, suffices. In that case the corresponding averaged constants (given by \eqref{eq:bar-const}) are of course given by $\bar{\mfk{m}}_{\chi,p} = \mfk{m}_{\chi}^p$, and are therefore trivially finite.   \vs{0.15cm}

\item One common example of $L_n$ is $L_n(x,y) \equiv L_n(x) = x$ or $U_nx$ for some unitary operator $U_n$. If $L_n(x,y) \equiv L_n(x)$, then centered $L_n$, that is, $\bar L_n \equiv 0$, and the condition on the corresponding growth exponent $l_1$ is trivially satisfied.\vs{.15cm}

 \item Clearly, $f_1\leq f_0$, where recall that $f_1$ and $f_0$ are the growth rates of $\bar F_n(x,y) = F_n(x,y)-P_xF_n(x,\cdot)(y)$ (centered $F_n$) and $F_n$, respectively. In some models, without any other information or suitable estimates on $\bar F$,  $f_1$ may just have to be taken the same as $f_0$,  in which case condition \ref{H-SS-exp} implies that the above result on uniform bounds on moments applies to  systems for which $f_0 <1/2.$ (and not $(1+\g)/2$). However, in some other models the optimal growth rate $f_1$ of $\bar F_n$ can indeed be lower than that of $F_n$. For example, as we noted before for the function $L_n$, if $F_n(x,y) \equiv F_n(x)$,  then  $\bar F_n(x,y) \equiv 0$ (that is, in particular, $f_1=0$), and this along with Theorem \ref{MP-stab} leads to Corollary \ref{cor-erg-mark} about Harris ergodicty of a large class of multiplicative Markovian systems. \vs{.15cm}

%\item  If $\mu_n\doteq\EE(\xi_n) \neq 0$, then the assertion of Proposition \ref{prop:switch} is still valid if we assume that 
%the growth rate of $G_n$, $g_0 < \g \wedge 1/2$. To see this, observe that \eqref{SS-jump} holds as before. Next,
% letting  $\bar\xi_{n} = \xi_{n}-\mu_{n}$, write \eqref{eq:ss-2} as
%$$X_{n+1} = L_n(X_n, Y_{n+1})+ F_n(X_n, Y_{n+1}) +  G_n(X_n,Y_{n+1})\mu_{n+1}+G_n(X_n,Y_{n+1})\bar\xi_{n+1},$$ 
%and observe that since  $g_0 < \g \wedge 1/2$,
%\begin{align*}
%\|G_n(X_n,Y_{n+1})\mu_{n+1}\|^2+2|\<G_n(X_n,& Y_{n+1})\mu_{n+1},L_n(X_n, Y_{n+1})+F_n(X_n, Y_{n+1}) \>|\\
%& = O(\|x\|^{(2g_0) \vee (1+g_0)\vee (\g_0+g_0)} = o(\|x\|^{1+\g}).
%\end{align*}	
%It now follows from the proof of Proposition \ref{prop:switch},  that \eqref{SS-neg-drft} also holds. \vs{.15cm} 

\end{itemize}	
}
\end{remark}

\begin{proof}[Proof of Proposition \ref{prop:switch}]	
	 Besides the different parameters in Assumption \ref{assum-SS2}, other  constants appearing in various estimates below will be denoted by $\mfk{m}_i$'s. They will not depend on $n$ but may depend on the parameters of the system.
	 
		For the proof we will only consider the case of \ref{H-SS-exp}-(a), where $f_0 < (1+\g)/2$; the proofs in the cases of \ref{H-SS-exp}-(b) and the second point in Remark \ref{rem-swsys} follow from \eqref{eq:xn2} and some minor modification of the arguments. For each $n$, define the functions $\hat G_n:\H\times \metsp\rt \H$ and $\tilde G, \ \bar G_n:\H\times \metsp\times\B \rt \H$  by
\begin{align*}
&\ \hat G_n(x,y) =\ \int_\B G_n(x,y,z) \nu_{n+1}(dz), \quad \tilde G_n(x,y,z) = G_n(x,y,z) - \hat G_n(x,y), \quad \text{ and}\\
&\ \bar G_n(x,y,z) =\ G_n(x,y,z) - P_{n,x}\hat G_n(x,\cdot)(y) = G_n(x,y,z) - \EE(G(X_n,Y_{n+1}, \xi_{n+1})|(X_n,Y_n)=(x,y))
\end{align*}
(recall that  $\nu_n$ is the distribution measure of $\xi_n$), and notice that by \ref{H-SS-FG} and \ref{H-SS-noise} for any $p>0$,
\begin{align}
\non
\EE\lf[|\hat G_n(X_n,Y_{n+1})|^p | \sigalg_n\ri] =&\ \int_{\metsp} \lf( \int_\B G_n(x,y,z) \nu_{n+1}(dz)\ri)^pP_{n, X_n}(Y_n,dy) \\
\non
\leq&\ \int_{\metsp} \int_\B \mfk{m}_G^p(y)(1+\|X_n\|)^{pg_0}\Psi(z)^{p} \nu_{n+1}(dz)P_{n, X_n}(Y_n,dy)  \\
\leq&\ \bar{\mfk{m}}_{\hat G,p}(1+\|X_n\|)^{pg_0},
\label{eq:Ghat-est}
\end{align}
where $\bar {\mfk{m}}_{\hat G,p} = \bar{\mfk{m}}_{G,p}m_*^p$ (recall $m_*^p=\sup_k\EE\lf[\Psi(\xi_k)^{p}\ri] < \infty$). It now easily follows that $\bar G_n$ and $\tilde G_n$ satisfy the following growth conditions:
\begin{align*}
\| \tilde G_n(x,y,z)\| \leq  \mfk{m}_{\tilde G}(y)(1+\|x\|)^{g_0}\Psi(z), \quad \text{and} \quad \| \bar G_n(x,y,z)\| \leq  \mfk{m}_{\bar G}(y)(1+\|x\|)^{g_0}\Psi(z)
\end{align*}
for some functions $\mfk{m}_{\bar G}(y)$ and  $\mfk{m}_{\tilde G}(y)$ (depending on $y$), where $\bar{\mfk{m}}_{\chi,p} < \infty$ for $\chi = \tilde G, \bar G$ (see \eqref{eq:bar-const} for definition of $\bar{\mfk{m}}_{\chi,p}$). Consequently, for any $p>0$
\begin{align*}
%\label{eq:G-est}
%\begin{aligned}
\EE\lf[\|\tilde G_n(X_n, Y_{n+1},\xi_{n+1})\|^p\big | \sigalg_n\ri] \leq&\ \bar{\mfk{m}}_{\tilde G,p}m_*^p(1+\|X_n\|)^{pg_0}, \\
 \EE\lf[\|\bar G_n(X_n, Y_{n+1},\xi_{n+1})\|^p\big | \sigalg_n\ri] \leq&\ \bar{\mfk{m}}_{\bar G,p}m_*^p(1+\|X_n\|)^{pg_0}.
%\end{aligned}
\end{align*}
 Also, 
\begin{align}
\label{eq:LF-est}
\begin{aligned}
\EE\lf[\|L_n(X_n, Y_{n+1})\|^2\big | \sigalg_n\ri] \leq&\  \|X_n\|^2 + 2\bar{\mfk{m}}_{ L,2,2}^{1/2}\|X_n\|+\bar{\mfk{m}}_{ L,2,2} = \lf(\bar{\mfk{m}}_{ L,2,2}^{1/2}+\|X_n\|\ri)^{2}\\
\EE\lf[\|F_n(X_n, Y_{n+1})\|^2\big | \sigalg_n\ri] \leq&\ \bar{\mfk{m}}_{F,2}(1+\|X_n\|)^{2f_0}.
\end{aligned}
\end{align}
%and similar growth estimates hold for $\EE\lf[\|\hat G_n(X_n, Y_{n+1})\|^p\big | \sigalg_n\ri]$ and  $\EE\lf[\|\tilde G_n(X_n, Y_{n+1}, \xi_{n+1})\|^p\big | \sigalg_n\ri]$
Now 	writing $G(X_n,Y_{n+1}, \xi_{n+1}) = \hat G_n(X_n,Y_{n+1})+\tilde G(X_n,Y_{n+1}, \xi_{n+1})$, we have
		\begin{align*}
		\|X_{n+1}\|^2 = & \ \|L_n(X_n,Y_{n+1})\|^2+ \|F_n(X_n,Y_{n+1})\|^2+  \|\hat G_n(X_n,Y_{n+1})\|^2 + \|\tilde G(X_n,Y_{n+1}, \xi_{n+1})\|^2 \\
		& \   + 2 \<L_n(X_n, Y_{n+1}),  F_n(X_n,Y_{n+1})\> 	+	2 \<(L_n+F_n+\hat G_n)(X_n, Y_{n+1}), \tilde G(X_n,Y_{n+1},\xi_{n+1})\>\\
		&+ 2 \<(L_n+F_n)(X_n, Y_{n+1}),  \hat G_n(X_n,Y_{n+1})\>.
	\end{align*}
Denoting the term $\<(L_n+F_n+\hat G_n)(X_n, Y_{n+1}), \tilde G(X_n,Y_{n+1},\xi_{n+1})\>$ by $J_{n+1}$ , we have 
\begin{align*}
\EE\lf[J_{n+1}| \sigalg_n\ri] =&\  \int_{\B}\int_{\metsp}\<(L_n+F_n+\hat G_n)(X_n, y), \tilde G(X_n,y,z)\> P_{n,X_n}(Y_n, dy)\nu_{n+1}(dz)\\
 =&\ \int_{\metsp}\lf\<(L_n+F_n+\hat G_n)(X_n, y), \int_{\B}\tilde G(X_n,y,z) \nu_{n+1}(dz)\ri\> P_{n,X_n}(Y_n, dy) =0.
\end{align*}
Also by Cauchy-Schwartz inequality, \eqref{eq:Ghat-est} and \eqref{eq:LF-est}
\begin{align*}
\EE\lf[|\<F_n(X_n, Y_{n+1}), \hat G_n(X_n, Y_{n+1})\>| \big | \sigalg_n\ri] \leq&\ \lf(\EE\lf[\|F_n(X_n, Y_{n+1})\|^2\big | \sigalg_n\ri]\ri)^{1/2}\lf(\EE\lf[\| \hat G_n(X_n, Y_{n+1})\|^2\big | \sigalg_n\ri]\ri)^{1/2}\\
\leq&\  \bar{\mfk{m}}_{F,2}^{1/2}\bar{\mfk{m}}_{\hat G,2}^{1/2}(1+\|X_n\|)^{f_0+g_0},
\end{align*}
and similarly, %for some constant $\breve{\mfk{m}}_{LG}$,
\begin{align*}
\EE\lf[|\<L_n(X_n, Y_{n+1}), \hat G_n(X_n, Y_{n+1})\>| \big | \sigalg_n\ri] \leq&\ \bar{\mfk{m}}_{\hat G,2}^{1/2}\lf(\bar{\mfk{m}}_{ L,2,2}^{1/2}\vee 1+\|X_n\|\ri)^{1+g_0}.
\end{align*}
Hence, on $\{\|X_n\|>B\}$
\begin{align}
	\non
\EE\lf[\|X_{n+1}\|^2 | \sigalg_n\ri] \leq &\ %P_{n,x_n}\|L_n(x_n,\cdot)\|^2(y_n)  + P_{n,x_n}\|F_n(x_n,\cdot)\|^2(y_n) + P_{n,x_n}\|G_n(x_n,\cdot)\|^2_{op}(y_n) m^2_* - 2 \mfk{m}_0\|x_n\|^{1+\g} \\
 \|X_n\|^2 +2\bar{\mfk{m}}_{ L,2,2}^{1/2}\|X_n\|+\bar{\mfk{m}}_{ L,2,2}+ \bar{\mfk{m}}_{F,2}(1+\|X_n\|)^{2f_0}+(\bar{\mfk{m}}_{\hat G,p}+\bar{\mfk{m}}_{\tilde G,p}m_*^p)(1+\|X_n\|)^{2g_0}\\
 &\  - 2\mfk{m}_0\|X_n\|^{1+\g} + 2\bar{\mfk{m}}_{F,2}^{1/2}\bar{\mfk{m}}_{\hat G,2}^{1/2}(1+\|X_n\|)^{f_0+g_0}+2\bar{\mfk{m}}_{\hat G,2}^{1/2}\lf(\bar{\mfk{m}}_{ L,2,2}^{1/2}\vee 1+\|X_n\|\ri)^{1+g_0}.
 \label{eq:xn2}
 \end{align} 
Since $\delta_0\doteq2(f_0 \vee g_0)\vee(f_0+g_0)\vee(1+g_0) < 1+\g,$ by \ref{H-SS-exp} it follows from the above inequality that we can choose $C>B$  large enough so that for $\|x_n\| >C$,
 \begin{align*}
 	\EE\lf[\|X_{n+1}\|^2 -\|X_n\|^2 \lf |\ri. \SC{F}_n\ri]
 \leq&\ \mfk{m}_1\lf(\|X_n\|^{\delta_0} - \|X_n\|^{1+\g}\ri) < 0.
\end{align*}

 Also
notice that choosing $C>B\vee1$ we have for $\|X_n\| >C$
\begin{align*}
\sqrt{\EE\lf[\|X_{n+1}\|^2 | \sigalg_n\ri]} + \|X_n\| \leq \mfk{m}_2 (1+\|X_n\|)^{1\vee\delta_0/2}  \leq 2^{1\vee\delta_0/2}\mfk{m}_2 \|X_n\|^{1\vee\delta_0/2}.
\end{align*}
Therefore for $\|X_n\| >C$,
\begin{align*}
\EE\lf[\|X_{n+1}\| | \sigalg_n\ri] - \|X_n\| \leq &\ \sqrt{\EE\lf[\|X_{n+1}\|^2 | \sigalg_n\ri]} - \|X_n\|
= \f{\EE\lf[\|X_{n+1}\|^2\lf |\ri. \SC{F}_n\ri] -\|X_n\|^2}{\sqrt{\EE\lf[\|X_{n+1}\|^2 | \sigalg_n\ri]} + \|X_n\|}\\
\leq &\  \mfk{m}_3\lf(\|X_n\|^{\delta_0-1\vee\delta_0/2} - \|X_n\|^{1+\g-1\vee\delta_0/2}\ri).
\end{align*}
Because of assumption \ref{H-SS-exp}, notice that 
$$
\|x\|^{\delta_0-1\vee\delta_0/2}-\|x\|^{1+\g-1\vee\delta_0/2} \stackrel{\|x\| \rt \infty}\Rt \begin{cases}
-\infty, & \quad \g>0\\
-\mfk{m}_3, & \quad \g=0.
\end{cases}
$$  
In either case, there exist a constant $A>0$, and a sufficiently large $C$, such that
\begin{align}\label{SS-neg-drft}
\EE\lf[\|X_{n+1}\| | \sigalg_n\ri] - \|X_n\| \leq & -A, \quad \text{ on } \|X_n\| >C.
\end{align}
Next, notice that 
\begin{align*}
\Big|\|X_{n+1}\| - \EE\lf[\|X_{n+1}\|\big | \sigalg_n\ri]\Big| \leq &\ \Big|\|X_{n+1}\|  - \|\EE[X_{n+1} \big| \sigalg_n]\|\Big| + \Big|\|\EE[X_{n+1} | \sigalg_n]\| - \EE\lf[\|X_{n+1}\| \big| \sigalg_n\ri]\Big|\\
\leq &  \|X_{n+1} - \EE\lf[X_{n+1}\big | \sigalg_n\ri]\|+ \Big|\EE\lf[\|X_{n+1}\| -\|\EE[X_{n+1} \big| \sigalg_n]\| \big| \sigalg_n \ri] \Big|\\
\leq &  \|X_{n+1} - \EE[X_{n+1} | \sigalg_n]\|+ \EE\lf[\|X_{n+1}-\EE[X_{n+1} | \sigalg_n]\| \big| \sigalg_n \ri].
\end{align*}
Hence,
\begin{equation}
	\label{SS-jump}
	\begin{aligned}
\jdiff_n = &\ \EE\lf[\lf|\|X_{n+1}\| - \EE\lf[\|X_{n+1}\|\big | \sigalg_n\ri]\ri|^p \Big| \sigalg_n\ri] \leq\ 2^p \EE\lf[\|X_{n+1} - \EE[X_{n+1} | \sigalg_n]\|^p\Big| \sigalg_n\ri]\\
=&\  2^p \EE\lf[\| \bar L(X_n, Y_{n+1})+\bar F(X_n, Y_{n+1})+\bar G(X_n,Y_{n+1}, \xi_{n+1})\|^p\Big| \sigalg_n\ri]\\
\leq&\  \mfk{m}_4 (1+\|X_n\|)^{p(l_1 \vee f_1 \vee g_0)}\equiv \phi_p(X_n),
\end{aligned}
\end{equation}
where $\phi_p(x) \doteq \mfk{m}_4 (1+\|x\|)^{p(l_1 \vee f_1 \vee g_0)}$. Since  $l_1 \vee f_1 \vee g_0<1/2$, for large enough $p$, we have $p(l_1 \vee f_1 \vee g_0) < p/2-1$. 
%Hence for any $r \in (0, \pow(s=p(g_0 \vee \g_1),p))$ we have $\phi_p(x)/\|x\|^{r} \rt 0$ as $\|x\| \rt \infty$.  
It now follows from Theorem \ref{t:maintheorem} (using $V(x) = \|x\|$) that for any $r \in (0, \pow(s=p(l_1 \vee f_1 \vee g_0),p))$, $\sup_n \EE\|X_n\|^{r} < \infty$. Since $p>0$ is arbitrarily large, the assertion follows. 

If $G_n(x,y,z)$ are centered, that is, if $\hat G_n\equiv 0$, then of course $\bar{\mfk{m}}_{\hat G,p}$ can be taken to be $0$ for all $p>0$, and from \eqref{eq:xn2} , $\delta_0=2(f_0\vee g_0)$. Consequently, we do not need $g_0<\g$ to have $\delta_0<1+\g$.

%For the second part, where $\mu_n\doteq\EE(\xi_n) \neq 0$, observe that \eqref{SS-jump} holds as before. Next,
% letting  $\bar\xi_{n} = \xi_{n}-\mu_{n}$, write \eqref{eq:ss-2} as
%$$X_{n+1} = L_n(X_n, Y_{n+1})+ F_n(X_n, Y_{n+1}) +  G_n(X_n,Y_{n+1})\mu_{n+1}+G_n(X_n,Y_{n+1})\bar\xi_{n+1},$$ 
%Since  $g_0 < \g \wedge 1/2$,
%\begin{align*}
%\|G_n(X_n,Y_{n+1})\mu_{n+1}\|^2+2|\<G_n(X_n,& Y_{n+1})\mu_{n+1},L_n(X_n, Y_{n+1})+F_n(X_n, Y_{n+1}) \>|\\
%& = O\lf(\|x\|^{(2g_0) \vee (1+g_0)\vee (\g_0+g_0)}\ri) = o\lf(\|x\|^{1+\g}\ri),
%\end{align*}	
%and from the same arguments it is now clear that \eqref{SS-neg-drft} also holds.

\end{proof}

\begin{corollary}\label{cor-erg-mark}
	Consider the class of $\{\sigalg_n\}$-adapted Markov processes taking values in $\R^d$, whose dynamics is defined by 
	\begin{align}\label{M-syst}
	X_{n+1} = L(X_n) + F(X_n) + G(X_n)\xi_{n+1},
	\end{align}
	where $F, L:\R^d \rt \R^d$, $G: \R^d \rt \M^{d\times d'}$ are continuous functions, and $d \leq d'$.
Assume that	

\begin{enumerate}[label={\rm (M-\arabic*)}, leftmargin=*, align=left, start=1]
	\item \label{M-LFG} $F$, $G$ and $L$ satisfy the growth conditions (a)\ $\|L(x)\| \leq \|x\|$ for $\|x\| >B$,  (b)\  $\dst \|F(x)\| \leq \mfk{m}_F(1+\|x\|)^{\g_0},$ and  (c)\ $ \|G(x)\| \leq \mfk{m}_G(1+\|x\|)^{g_0};$  
	
	%	for $\|x\| >B$ and for all $y$;\\
%	
%	\item \label{M-FG} $F$, $G$ satisfy the growth conditions: $\dst \|F(x)\| \leq \mfk{m}_F(1+\|x\|)^{\g_0},$ and  $ \|G(x)\| \leq \mfk{m}_G(1+\|x\|)^{g_0};$  

	\item \label{M-ND} for some constant $\mfk{m}_0, B$ and exponent $\g\geq 0$, 
	$$\<F(x),L( x)\> \leq - \mfk{m}_0\|x\|^{1+\g}, \quad \text{ for } \|x\| >B;$$
	
	\item \label{M-exp} the exponents satisfy: (a) 	  $\g_0 < (1+\g)/2$, or \ $\g_0 = (1+\g)/2$ and $\mfk{m}_{F} \leq \mfk{m}_0/2$; (b) $g_0 <1/2$;

	\item  the $\xi_{n}$ are i.i.d $\R^{d'}$-valued random variables with density $\rho $ with respect to Lebesgue measure, $\l_{\text{leb}}$; $\rho(z)>0$ for all $z \in \R^{d'}$, $\sup_{z\in \R^{d'}}\rho(z) < \infty$, and for each $p>0$, $m_*^p = \EE(\|\xi_1\|^p) < \infty$; 
	
	\item \label{cond-nondeg} 	 for some $\theta \geq 0$ and $\vep_0>0$, 
	$$u^TG(x)G(x)^Tu \geq  \vep_0 u^Tu / (1+\|x\|)^\theta, \quad \forall \ u, x \in \R^d.$$

 \end{enumerate}	

If in addition $\EE(\xi_1) = 0$, for all $n$,  then (a)\ $\{X_n\}$ is PHR and aperiodic with a unique invariant distribution $\pi$,  
	(b)\ $\sup_n \EE_{x_0}(\|X_n\|^r) \vee \EE_{\pi}\|X_n\|^r <\infty$, and (c)\ \eqref{eq:conv-1} or equivalently, \eqref{eq:conv-2} holds with $V(u) = \|u\|$,
	for any $x_0$ and $r >0$. If $\EE(\xi_1) \neq 0$, then the same assertion is true provided $g_0<\g \wedge 1/2$

\end{corollary}	

\begin{proof}
	Since $L,F$ and $G$ are continuous, it follows by the dominated convergence theorem that $\{X_n\}$ is weak-Feller.
From the assumption \ref{cond-nondeg}, it follows that
$GG^T$ is positive definite (in particular, non singular), and det$(G(x)G^T(x)) \geq \vep^d_0/(1+\|x\|)^{\theta d}$. Note that
$\tk(x,\cdot)$ admits a density $q(x,\cdot)$. Specifically,
$$q(x,y) = \f{1}{\sqrt{\text{det}(G(x)G^T(x))}}\rho\lf(G(x)^{-}_R(y-L(x) - H(x)\ri) \leq \sup_{z} \rho(z)(1+\|x\|)^{\theta d/2}/\vep^{d/2}_0,$$
where $G(x)^-_R = G^T(x)\lf(G(x)G(x)^T\ri)^{-1}$ is the Moore-Penrose pseudoinverse (in particular, right inverse) of $G(x)$.   Moreover, since $\rho(z) > 0$ a.s, 
for each $x$, $q(x,y) >0$ a.s in $y$ (with respect to $\l_{\text{leb}}$), and consequently, $X_n$ is $\l_{\text{leb}}$-irreducible. This shows that Condition \ref{MP-stab-iv} of Theorem \ref{MP-stab} holds.
The various assertions now follow from Theorem \ref{MP-stab} and Proposition \ref{prop:switch} 
\end{proof}	

\begin{remark}{\rm
The condition \ref{cond-nondeg} is much weaker than uniform ellipticity condition that is sometimes imposed on $GG^T$ for these kinds of models - the latter requiring for some $\vep_0>0$, $u^TG(x)G(x)^Tu \geq  \vep_0 u^Tu$, for all $u, x \in \R^d.$

The above theorem also holds, with some possible minor modifications, for systems of the form \eqref{M-syst} taking values in other  locally compact spaces with $\xi_n$ admitting a density $\rho$ with respect to the Haar measure. In particular, for such systems taking values in a countable state space like $\Z^d$ or $\rat^d$, notice that the transition probability mass function (density with respect to counting measure) $q(x,y)$ naturally exists and $q(x,y) \leq 1$, that is, the bound on $q$ in condition \ref{MP-stab-iv} of Theorem \ref{MP-stab} is trivially satisfied. Hence 
condition \ref{cond-nondeg} in Corollary \ref{cor-erg-mark} is not needed in this case. However, depending on the specific model, one might still require $G$ to have full row rank for establishing irreducibility of the chain.

}
\end{remark}	

As an important application, the above corollary can be used to establish ergodicity of numerical schemes of stochastic differential equations (SDEs).
\begin{example} {Euler-Maruyama scheme for ergodic SDEs:}
{\rm
Consider the SDE
\begin{align*}
X(t) = X(0) + \int_0^t F(X(s)) ds + \int_0^t G(X(s)) dW(s),
\end{align*}
and suppose that $X$ is ergodic with invariant / equilibrium distribution $\pi$ - which is typically unknown. Approximating this equilibrium distribution  is an important computational problem in various areas including statistical physics, machine learning, mathematical finance etc. Since numerically solving the corresponding (stationary) Kolomogorov PDE for $\pi$  is computationally expensive even when the dimension is as low as $3$,  one commonly resorts to discretization schemes like the Euler-Maruyama method:
\begin{align*}
X^\Delta(t_{n+1}) = X^\Delta(t_{n})  + F(X^\Delta(t_{n}) ) \Delta +   \Delta^{1/2} G(X^\Delta(t_{n}) ) \xi_{n+1}.
\end{align*}
Here the $\xi_n$ are iid $N(0,I)$-random variables, and $\{t_n\}$ is a partition of $[0,\infty)$ with $t_{n+1}-t_n = \Delta$ - the step size of discretization. However, the use of such discretization techniques in approximating $\pi$ is justified provided one can establish (a)  ergodicity of the discretized chain $\{X^\Delta(t_n)\}$ with a unique invariant distribution $\pi^\Delta$, and (b) convergence of $\pi^\Delta$ to $\pi$ as $\Delta \rt 0$. This is a hard problem involving infinite time horizon, and usual error analysis of Euler-Maruyama schemes, which has of course been well studied in the literature, is not useful here, as they are over finite time intervals. In comparison, much less is available on theoretical error analyses of these types of infinite-time horizon approximation problems, and some important results in this direction have been obtained by Talay \cite{Tatu90, Tal90, GrTa13}. A recent paper \cite{GS18} (also see the references therein for more background on the problem) conducts a thorough large deviation error analysis of the problem in an appropriate scaling regime. 

This short example do not attempt to address both the points (a) and (b) of this problem as that requires a separate paper-long treatment. Here, we are only interested in the point (a) above - which is ergodicity of the  discretized chain $\{X^\Delta(t_n)\}$. It is well known that ergodicity of $X$ does not guarantee the ergodicity of the discretized chain $X^\Delta$. Discretization can destroy the underlying Lyapunov structure of an ergodic SDE!

In \cite{Tatu90, Tal90} among several other important results, Talay et al. in particular showed that the chain $\{X^\Delta(t_n)\}$ is ergodic with unique invariant measure $\pi^\Delta$ and $\EE(f(X^\Delta(t_n)) \rt \pi^\Delta(f)$ as $n\rt \infty$ for any $f \in C^\infty(\R^d,\R)$ such that $f$ and all its derivatives have polynomial growth  under the assumption (i) $\<F(x), x\> \leq -\mfk{m}_0 \|x\|^2$, for $\|x\| >B$, (ii)  $F$ and $G$ are $C^\infty$ with bounded derivatives of all order and (iii) $GG^T$ is uniformly elliptic and bounded. An application of  Corollary \ref{cor-erg-mark} shows that this result can be significantly improved with stronger convergence results under weaker hypothesis (c.f \ref{M-LFG} -\ref{cond-nondeg}). In particular, uniform ellipticity and boundedness conditions on $GG^T$, which are quite restrictive for many models, can be removed.
}
\end{example}

\subsection{Moment stability of linear stochastic control systems}
Consider the system 
\begin{align}\label{eq:cntrl}
X_{n+1} = AX_{n} + Bu_n + \xi_{n+1}
\end{align}
 We are interested in the problem of finding conditions under which a linear stochastic system with possibly unbounded additive stochastic noise is globally stabilizable with bounded control inputs $\{u_n\}$.
Stabilization of stochastic linear systems with bounded control is a topic of significant interest in control engineering because of its importance in diverse fields; suboptimal control strategies such as receding-horizon control, and rollout algorithms, among others, can be easily constructed incorporating  such constraints, and have become popular in applications. Here we simply refer to \cite{ref:RamChaMilHokLyg-10} and references therein for a detailed background on this topic.

  Of course, boundedness of some moments of the noise component is necessary for attaining (moment) stability of the system. Specifically, we consider the following problem:

\np
{\em Problem:} 
Suppose $\mathbb{U} \doteq \lf\{z\in\R^m : \|z\| \leq U_{\max}\ri\}$. We consider admissible possible $k$-history dependent control policies of the type $\pi = \{\pi_n\}$ so that $\pi_n:\R^{d\times k} \rt \mathbb{U}$, and for every $ y_1, y_2 \hdots, y_k \in \R^{d}$, 
$ \pi_n(y_1, \hdots, y_k) \in \mathbb{U}$. Given $r \geq 1$ and $U_{\max} > 0$, find an admissible policy $\pi=\{\pi_n\}_{n \in\N}$ with control authority $U_{\max}$, such that the system \eqref{eq:cntrl} with 
$u_n = \pi_n(X_{n-k+1}, \hdots,X_{n-1}, X_n)$ is $r$-th moment stable, that is, for every initial condition $X_0=x_0$, $\sup_{n} \EE_{x_0}\|X_n\|^r < \infty$.
%			\begin{quote}
%				for every initial condition $x_0 = \xz\in\R^d$ there exists a constant $\msbound > 0$ such that the closed-loop system \eqref{e:chclosedloop} satisfies $\EE_{\xz}\bigl[\norm{x_t}^r\bigr] \le \msbound$ for all $t\in\Nz$.
%			\end{quote}

It is known that mean square boundedness holds for systems with bounded controls where  $A$ is Schur stable, that is, all eigenvalues of $A$ are contained in the open unit disk (the proof uses Foster-Lyapunov techniques from \cite{ref:MeyTwe-09}). In the more general framework, under the assumption that the pair $(A, B)$ is only stabilizable (which in particular allows the eigenvalues of $A$ to lie on the closed unit disk), \cite{ref:RamChaMilHokLyg-10} shows that there exist  a $k$-history dependent control policy that ensures moment stability of \eqref{eq:cntrl}, provided the control authority $U_{\max}$ is chosen sufficiently large. It was conjectured in \cite{ref:RamChaMilHokLyg-10}, that the lower bound on $U_{\max}$ can possibly be lifted with newer techniques, and here we demonstrate that is indeed the case. The following result is an easy corollary of Proposition \ref{prop:switch}. For simplicity, we assume that $A$ is orthogonal and $(A, B)$ is reachable in $k$-steps. The steps from there to the more general case are similar to that in \cite{ref:RamChaMilHokLyg-10}.  In case $B$ has full row rank, it will follow that $k$ can be taken to be $1$, that is, the resulting policy is stationary feedback. 

\begin{proposition}\label{prop:control}
Consider the system defined by \eqref{eq:cntrl}. Suppose that $A$ is orthogonal and the pair $(A, B)$ is reachable in $k$ steps  (that is, $\rank(\SC{R}_k) = d$, where $\SC{R}_k = [B \ AB\ A^2B\ \hdots \ A^{k-1}B]$). Then for any $U_{\max} >0$, there exists a $k$-history dependent policy $\pi = \{\pi_n\}$ such that  given $(X_{n-k+1}, \hdots,X_{n-1}, X_n) = (x_{n-k+1}, \hdots,x_{n-1}, x_n)$, 
$\pi_n(x_{n-k+1}, \hdots,x_{n-1}, x_n)  \doteq f_{n \md k}(x_{\floor{n/k}k})$
for some functions $f_0,f_1,\hdots, f_{k-1}:\R^d \rt \R^m$ where $\|f_i(x)\| \leq U_{\max}$ for $i=0,1,2,\hdots,k-1$, and for which $\sup_{n} \EE_{x_0}\|X_n\|^r < \infty$ for any $x_0 \in \R^d$.
\end{proposition}

\begin{proof}
Define $\hat X^{(k)}_n = X_{nk}$, and notice that by iterating \eqref{eq:cntrl} we get
\begin{align*}
\hat X^{(k)}_{n+1} = & A^k\hat X^{(k)}_n + \SC{R}_k\begin{pmatrix} u_{(n+1)k-1}\\ \vdots\\ u_{nk+1}\\ u_{nk}\end{pmatrix}+\sum_{j=1}^{k}A^{k-1-j}\xi_{nk+i}\equiv A^k \hat X^{(k)}_n + \SC{R}_k \hat u^{(k)}_{n}+\hat \xi^{(k)}_n
\end{align*}
Notice that $\EE( \xi^{(k)}_n) = 0$ and $\sup_n \EE\| \xi^{(k)}_n\|^p \leq \hat \const_k $ for some constant $ \hat \const_k >0$. 
Since $ \SC{R}_k$ has full row rank, it has a right inverse $ \SC{R}_k^{-}$. Define
$$\sat(y) = \begin{cases}
y, & \quad  y \in B(0,\hat U_{\max})\\
\hat U_{\max} \ y/\|y\|,& \quad \text{otherwise}
\end{cases}
$$
 and choose $\hat u^{(k)}_{n} = -\SC{R}_k^{-}A^k\sat( \hat X^{(k)}_n)$,   where 
 $\hat U_{\max}$ is such that $\|\SC{R}_k^{-}A^k\|\hat U_{\max} \leq U_{\max}.$
 This yields the system
% That is the controls $\{u_{n}\}$ are given by 
%
 %$u_{nk+r} = - e_{k-r}^T \SC{R}_k^{-}(A^{T})^k\sat ( X_{nk}) $
 \begin{align*}
\hat X^{(k)}_{n+1} =A^k \hat X^{(k)}_n - A^k\sat( \hat X^{(k)}_n)+\hat \xi^{(k)}_n.
\end{align*}
Since for $\|z\| > U_{\max}$,  $\<A^kz, -A^k\sat(z) \> =  -\|z\|$ (recall that $A$ is orthogonal), we have from Proposition \ref{prop:switch} that there exists a constant $ \mfk{c}^{(k,r)}_0$ such that 
$$\sup_{n} \EE\|\hat X^{(k)}_{n}\|^r = \sup_{n} \EE\| X_{nk}\|^r <  \mfk{c}^{(k)}_0.$$
 It is now immediate  by a sequential argument that for any $\ell=0,1,\hdots, k-1$, 
$
\EE\| X_{nk+\ell}\|^r \leq  \mfk{c}^{(k,r)}_\ell
$
where $ \mfk{c}^{(k,r)}_\ell = 3^{r-1}\lf(\|A\|^r \mfk{c}^{(k,r)}_{\ell-1} + \|\SC{R}_k^{-}A^k\|^rU^r_{\max} +\mfk{m}^r_*\ri)$.

 Notice that the original controls $u_n$ are defined 
 $$u_{n} = - E_{k-(n\md k)}^T \SC{R}_k^{-}A^k\sat ( X_{\floor{n/k}k}),$$
 where the matrices $E_j \in \M_{m \times km}, \ j=1,2, \hdots, k$, are defined by 
 \begin{align*}
 E_j = \begin{bmatrix}
\bm{0}_{m\times m}  & \hdots & \bm{0}_{m\times m} & \underbrace{ \bm{I}_{m\times m}}_{j\text{-th block}} & \bm{0}_{m\times m}& \hdots &\bm{0}_{m\times m}
 \end{bmatrix}
  \end{align*}
In particular, from the state at time $nk$,  the present and the next $k-1$ controls $u_j, j=nk, nk+1, \hdots, nk+k-1$ can be computed.
\end{proof}